\documentclass{amsart} 
\usepackage{amsthm, amsmath, mathtools}
\usepackage{amssymb}
\usepackage{amscd}
\usepackage{enumerate}
\usepackage{mathrsfs} 
\usepackage[curve,matrix,arrow,frame,tips]{xy}
\usepackage{colonequals}
\usepackage{verbatim}
\usepackage{pdfsync}
\usepackage{graphicx}
\usepackage[usenames,dvipsnames]{color}

\usepackage{url}

\usepackage{fullpage}

\usepackage[normalem]{ulem}

\numberwithin{equation}{section} 

\theoremstyle{plain}
\newtheorem{theorem}{Theorem}[section]
\newtheorem{proposition}[theorem]{Proposition}
\newtheorem{lemma}[theorem]{Lemma}

\newtheorem{corollary}[theorem]{Corollary}

\theoremstyle{definition}

\newtheorem{example}[theorem]{Example}
\newtheorem{hypo}[theorem]{Hypothesis}

\theoremstyle{remark}
\newtheorem{remark}[theorem]{Remark}

\renewcommand{\to}{\longrightarrow}

\usepackage{microtype}

\begin{document}

\title{On representations of $\mathrm{Gal}(\overline{\mathbb{Q}}/\mathbb{Q})$, $\widehat{GT}$ and $\mathrm{Aut}(\widehat{F}_2)$}

\dedicatory{In memory of Jan Saxl}

\author[F. Bleher]{Frauke M. Bleher}
\address{F.B.: Department of Mathematics\\University of Iowa\\
14 MacLean Hall\\Iowa City, IA 52242-1419\\ U.S.A.}
\email{frauke-bleher@uiowa.edu}

\author[T. Chinburg]{Ted Chinburg}
\address{T.C.: Department of Mathematics\\University of Pennsylvania\\Philadelphia, PA 19104-6395\\ U.S.A.}
\email{ted@math.upenn.edu}

\author[A. Lubotzky]{Alexander Lubotzky}
\address{A.L.:  Einstein Institute of Mathematics\\Hebrew University\\ 
Givat Ram, Jerusalem 91904\\ Israel}
\email{alex.lubotzky@mail.huji.ac.il}

\date{May 25, 2021}

\subjclass[2010]{Primary 14G32; Secondary 20F34, 14H30}

\begin{abstract} 
By work of Bely\u{\i} \cite{Belyi}, the absolute Galois group $G_{\mathbb{Q}}=\mathrm{Gal}(\overline{\mathbb{Q}}/\mathbb{Q})$ of the field $\mathbb{Q}$ of rational numbers can be embedded into $A=\mathrm{Aut}(\widehat{F}_2)$, the automorphism group of the free profinite group $\widehat{F}_2$ on two generators. The image of $G_{\mathbb{Q}}$ lies inside $\widehat{GT}$, the Grothendieck-Teichm\"uller group. While it is known that every abelian representation of $G_{\mathbb{Q}}$ can be extended to $\widehat{GT}$, Lochak and Schneps \cite{LochakSchneps} put forward the challenge of constructing irreducible non-abelian representations of $\widehat{GT}$. We do this virtually, namely by showing that a rich class of arithmetically defined representations of $G_{\mathbb{Q}}$ can be extended to finite index subgroups of $\widehat{GT}$. This is achieved, in fact, by extending these representations all the way to  finite index subgroups of $A=\mathrm{Aut}(\widehat{F}_2)$. We do this by developing a profinite version of the work of Grunewald and Lubotzky \cite{LG}, which provided a rich collection of representations for the discrete group $\mathrm{Aut}(F_d)$.
\end{abstract}

\maketitle

\section{Introduction}
\label{s:intro}

Let $G_{\mathbb{Q}}=\mathrm{Gal}(\overline{\mathbb{Q}}/\mathbb{Q})$ be the absolute Galois group of the field $\mathbb{Q}$ of rational numbers. Let $A=\mathrm{Aut}(\widehat{F}_2)$ be the automorphism group of the free profinite group $\widehat{F}_2$ on two generators. Bely\u{\i} showed in \cite{Belyi} that there is a canonical embedding $\iota$ of $G_{\mathbb{Q}}$ into $A$ (see \S \ref{s:GT} for details).  In \cite{D}, Drinfel'd defined the so-called Grothendieck-Teichm\"uller group whose profinite version $\widehat{GT}$ is a certain subgroup of $A$ that contains the image of  $G_{\mathbb{Q}}$ under $\iota$. For background on $\widehat{GT}$ and on variants of this group, see \cite{HarbaterSchneps, SchnepsSurvey} and \cite{LochakSchneps}.  Thus $\iota(G_{\mathbb{Q}})\le \widehat{GT}\le A$, and much effort has been dedicated to understanding the connection between $G_{\mathbb{Q}}$ and $\widehat{GT}$. It is known that every homomorphism from $G_{\mathbb{Q}}$ to an abelian group can be extended to $\widehat{GT}$ (see, for example, \cite[\S1.4(6)]{SchnepsSurvey}) but no similar result is known for non-abelian quotients. In \cite[p. 179]{LochakSchneps}, Lochak and Schneps wrote that ``no  irreducible non-abelian representation of any version of (the Grothendieck-Teichm\"uller group) has been constructed to date."  The goal of this paper is to address this challenge. Our main results will show that many naturally constructed representations of $G_{\mathbb{Q}}$ can be extended to $\widehat{GT}$, at least virtually, i.e. to a finite index subgroup. The main point of our work is that instead of trying to extend a representation from $G_{\mathbb{Q}}$ to $\widehat{GT}$, we extend it further all the way to a representation of a finite index subgroup of $A=\mathrm{Aut}(\widehat{F}_2)$, and thus in particular to a finite index subgroup of $\widehat{GT}$. What makes this possible is adapting, to the profinite world, the work of Grunewald and Lubotzky in \cite{LG}.  They constructed many  linear representations of the discrete group $\mathrm{Aut}(F_d)$ and showed that these provide surjections of $\mathrm{Aut}(F_d)$ onto various arithmetic groups. It turns out that the analogous profinite theory is easier than the discrete one (see below) and leads to stronger results. In particular, this approach works  for $A=\mathrm{Aut}(\widehat{F}_2)$, while the discrete theory has given interesting results for $\mathrm{Aut}(F_d)$ only when $d\ge 3$.  As an illustrative example, we will show the following result in connection with the challenge of Lochak and Schneps:

\begin{theorem}
\label{thm:irred}
Suppose $t> 1$ is an odd integer. Let $\zeta$ be a primitive $t^{th}$ root of unity, and let $E_\zeta$ be the elliptic curve with affine equation
$$y^2 = x (x-1) (x-\zeta).$$ 
There is a number field $F$ over which $E_\zeta$ is defined, with the following property for every prime $\ell$. The action of the finite index subgroup $G_F = \mathrm{Gal}(\overline{\mathbb{Q}}/F)$ of $G_{\mathbb{Q}}$ on the $\ell$-adic Tate module $T_\ell(E_\zeta)$ can be extended to an action of a finite index subgroup $A_{E_\zeta}$ of $A$ that contains the image of $G_F$ under the Bely\u{\i} embedding $\iota: G_{\mathbb{Q}}\to A$.  Let $\tau_{\zeta,\ell}: A_{E_\zeta}\to \mathrm{GL}(\mathbb{Q}_\ell \otimes_{\mathbb{Z}_\ell} T_\ell(E_\zeta))$ be the associated representation of $A_{E_\zeta}$ over $\mathbb{Q}_\ell$. Suppose the value $\phi(t)$ of Euler's phi function on $t$ is larger than $24$.  Then the restriction of $\tau_{\zeta,\ell}$ to any subgroup between $A_{E_\zeta}$ and $\iota(G_F)$ is absolutely irreducible and non-abelian.  This is true, in particular, for the restriction of $\tau_{\zeta,\ell}$ to the finite index subgroup $\widehat{GT} \cap A_{E_\zeta}$ of $\widehat{GT}$.
\end{theorem}

We now state a more general result concerning arbitrary smooth projective curves $X$ defined over $\overline{\mathbb{Q}}$. By a theorem of Bely\u{\i} \cite[Theorem 4]{Belyi1}, $X$ can be realized as a cover $\lambda:X\to\mathbb{P}^1_{\overline{\mathbb{Q}}}$ that is unramified outside $\{0,1,\infty\}$. Suppose $F$ is a number field over which $X$, the points in $\lambda^{-1}(\{0,1,\infty\})$ and the generalized Jacobian $J_\lambda(X)$ of $X$ with respect to  $\lambda^{-1}(\{0,1,\infty\})$ are defined.   Let $T_{\ell,\lambda}(X)$ denote the $\ell$-adic Tate module of $J_\lambda(X)$. This is a $\mathbb{Z}_\ell$-module of rank $2\,\mathrm{genus}(X)+|\lambda^{-1}(\{0,1,\infty\})|-1$.  The finite index subgroup $G_F=\mathrm{Gal}(\overline{\mathbb{Q}}/F)$ of $G_{\mathbb{Q}}$ acts naturally  on $T_{\ell,\lambda}(X)$ and  on the adelic Tate module
$$T_\lambda(X):=\prod_{\ell\;\mathrm{ prime}}T_{\ell,\lambda}(X)$$
of $J_\lambda(X)$. This gives rise to  Galois representations, which we denote by
$$\rho_{X,\ell}:G_F \to \mathrm{GL}(T_{\ell,\lambda}(X))
\quad \mbox{and} \quad
\rho_X: G_F \to \mathrm{GL}(T_\lambda(X)).$$

\begin{theorem}
\label{thm:main1}
There is a representation $\widetilde{\rho}_X: A_X\to \mathrm{GL}(T_\lambda(X))$ of a finite index subgroup $A_X$ of $A$ that agrees with $\rho_X$ on the finite index subgroup $G_F\cap \iota^{-1}(A_X)$ of $G_F$. In particular, the same holds for every $\rho_{X,\ell}$ that is extended to $\widetilde{\rho}_{X,\ell}: A_X\to \mathrm{GL}(T_{\ell,\lambda}(X))$.
\end{theorem}

To prove Theorem \ref{thm:main1}, let us recall the results in \cite{LG}, which are not needed here, but which were the inspiration for the current work. In that paper, the following situation was studied. Let $F=F_d$ be the free group on $d\ge 2$ generators and let $\pi:F_d\to H$ be an epimorphism onto a finite group $H$, with $\mathcal{R}=\mathrm{Ker}(\pi)$. Let $\mathcal{A}=\mathrm{Aut}(F_d)$ and $\mathcal{A}(\pi)=\{\alpha\in\mathcal{A} : \pi\circ\alpha=\pi\}$. The finite index subgroup $\mathcal{A}(\pi)$ of $\mathcal{A}$ preserves $\mathcal{R}$. Hence, we get a linear representation $\rho_\pi: \mathcal{A}(\pi)\to\mathrm{GL}(\overline{\mathcal{R}})$, where $\overline{\mathcal{R}} = \mathcal{R}/[\mathcal{R},\mathcal{R}]$ is a $\mathbb{Z}$-module of rank $|H|(d-1)+1$. The image of $\rho_\pi$ is inside the arithmetic group of all the $H$-module automorphisms of $\overline{\mathcal{R}}$. The main result of \cite{LG} is the claim that \textbf{under some technical condition on $\pi$} (``$\pi$ is redundant") $\rho_\pi(\mathcal{A}(\pi))$ is, after projectivization, an arithmetic group. This provides a rich class of arithmetic virtual quotients of $\mathcal{A}=\mathrm{Aut}(F_d)$.

In \S \ref{s:profinite}, we develop the analogous profinite theory for all $d \ge 1$. Let $\pi$ now be a continuous epimorphism  $\pi:\widehat{F}_d\to H$, and let $A(\pi)=\{\alpha\in\mathrm{Aut}(\widehat{F}_d) : \pi\circ\alpha=\pi\}$.  Let $R=\mathrm{Ker}(\pi)$. Then $\overline{R} = R/[R,R]$ is a $\widehat{\mathbb{Z}}[H]$-module giving rise to $\rho:A(\pi)\to\mathrm{GL}(\overline{R})$. The result here is stronger than in the discrete case and without any assumption: 

\begin{theorem}
\label{thm:main2}
For every $d\ge 1$ and every such $\pi$, the image $\rho(A(\pi))$ equals the subgroup $\mathrm{Aut}_{H,\beta}(\overline{R})$ of all $\widehat{\mathbb{Z}}[H]$-module automorphisms of $\overline{R}$ preserving the extension class $\beta$ of the short exact sequence
\begin{equation}
\label{eq:beta}
1 \to \overline{R} \to \widehat{F}_d/[R,R] \xrightarrow{\overline{\pi}} H \to 1 
\end{equation}
where $\overline{\pi}$ is induced by $\pi$.
\end{theorem}

See Theorem \ref{thm:main} below for an even more general result. This theorem is deduced from a powerful lemma of Gasch\"utz which holds for (pro)finite groups but not for discrete groups. This explains why Theorem \ref{thm:main2} does not hold in general for discrete groups; see also \cite[Prop. 9.7]{LG}.

Theorems \ref{thm:main1} and \ref{thm:main2} are connected with each other because of Bely\u{\i}'s theorem.  Let $H$ be the Galois group of the normal closure $\widetilde{\lambda}: Y \to \mathbb{P}^1_{\overline{\mathbb{Q}}}$ of a Bely\u{\i} cover $\lambda:X\to\mathbb{P}^1_{\overline{\mathbb{Q}}}$, and let $\mathcal{J}$ be the Galois group of $Y$ over $X$.  The \'etale fundamental group of $\mathbb{P}^1_{\overline{\mathbb{Q}}} - \{0,1,\infty\}$ with respect to a geometric base point is naturally isomorphic to the free profinite group $\widehat{F}_2$.  Letting $\pi:\widehat{F}_2 \to H$ be the homomorphism associated to $\widetilde{\lambda}$, we get from Theorem \ref{thm:main2} a lift to a finite index subgroup of $A(\pi)$ of the action of a subgroup of finite index in $G_{\mathbb{Q}}$ on $T_{\widetilde{\lambda}}(Y)$.  Taking $\mathcal{J}$-coinvariants  of this lift and then dividing by the finite torsion subgroup of these coinvariants leads to Theorem \ref{thm:main1}; see Theorem \ref{thm:restriction} for details.

One natural question is whether similar results hold if one replaces  the adelic Tate module $T_\lambda(X)$ of the generalized Jacobian of $X$ with respect to the ramification locus of $\lambda:X \to \mathbb{P}^1_{\overline{\mathbb{Q}}}$ with the adelic Tate module $T(X)$ of the Jacobian of $X$.  Using Theorem \ref{thm:main1}, we will give in \S \ref{s:JacobianAlone} a sufficient condition for the lift of the action on $T_\lambda(X)$ which we construct to descend to an action on $T(X)$; see Theorem \ref{thm:preserve2}.   If $X = Y$ then this condition is also necessary.  The condition is not always satisfied, but it is if, for example, the Galois group $H$ of the normal closure of $\lambda$ is abelian. Theorem \ref{thm:irred} is proved by applying Theorem \ref{thm:preserve2} to a suitable Bely\u{\i} cover when $X = E_\zeta$.  

Theorem \ref{thm:main2} gives the precise image of $A(\pi)$ in $\mathrm{GL}(\overline{R})$ and hence also in $\mathrm{GL}(T_\lambda(X))$. It is of interest to compare it with the image of $G_F$ in $\mathrm{GL}(T_\lambda(X))$. In \S \ref{s:examples} we analyze the case in which $X$ has genus $1$ and $\lambda:X \to \mathbb{P}^1_{\overline{\mathbb{Q}}}$ is Galois. 

We end this introduction with some comments and questions.  The main open question concerning the Grothendieck-Teichm\"uller group $\widehat{GT}$ is whether there is a variant of this group that is equal to the image of $G_{\mathbb{Q}}$ under the Bely\u{\i} embedding $\iota: G_{\mathbb{Q}}\to A$. If this is true then  every linear representation of $G_{\mathbb{Q}}$ must lift to the variant of $\widehat{GT}$.   In view of Theorems \ref{thm:irred} and \ref{thm:main1} above, it would be interesting to find other natural representations of finite index subgroups of $G_{\mathbb{Q}}$ that lift to representations of finite index subgroups of $A$.  For example, the Galois representations provided by the adelic Tate modules of the Jacobians of modular curves arise in the theory of modular forms of weight two for congruence subgroups of $\mathrm{PSL}_2(\mathbb{Z})$;  see, for example, \cite{ShimuraBook} and \cite{HidaModForms}.   It would be interesting if other Galois representations provided by modular forms of other weights, and more generally by $\Lambda$-adic modular forms, have lifts to finite index subgroups of $A$.   At a minimum, one would like to check whether such representations can be lifted to finite index subgroups of the Grothendieck-Teichm\"uller group $\widehat{GT}$. For related questions of this kind, see the end of  \cite[\S1]{LochakSchneps}.

\medbreak

\noindent {\bf Acknowledgements.} The first and second authors would like to thank the Hebrew University in Jerusalem for its support and hospitality during work on this paper.  The first author was supported in part by NSF FRG Grant No.\ DMS-1360621 and by NSF Grant No.\ DMS-1801328. The second author was supported in part by NSF FRG Grant No.\ DMS-1360767,  NSF SaTC Grants No. CNS-1513671 and CNS-1701785, and Simons Foundation Grant No.\ 338379. The third author is indebted for support from the NSF (Grant No.\ DMS-1700165)  and the European Research Council (ERC) under the European Union’s Horizon 2020 research and innovation program (Grant No.\ 692854). This material is based upon work supported by a grant from the Institute for Advanced Study.

\section{A construction of linear representations of $\mathrm{Aut}(\widehat{F}_d)$.}
\label{s:profinite}

Let $\pi:\widehat{F}_d \to H$ be a continuous epimorphism from $\widehat{F}_d$ to a finite group $H$, and let $R$ be the kernel of $\pi$.  Define $\overline{R} = R/[R,R]$ to be the maximal abelian quotient of $R$.  Then $\overline{R}$ is isomorphic to $\widehat{\mathbb{Z}}^{|H|(d-1)+1}$ where $\widehat{\mathbb{Z}}$ is the profinite completion of $\mathbb{Z}$. We consider surjections $R \to \widetilde{R}$ whose kernel $\widetilde{I}$ is a closed normal subgroup of $\widehat{F}_d$ containing $[R,R]$, so that $\widetilde{R}$ is a quotient of $\overline{R}$.  We will be interested in constructing large linear representations of the group
\begin{equation}
\label{eq:stab}
A(\pi,\widetilde{I})  = \{\alpha \in \mathrm{Aut}(\widehat{F}_d): \pi \circ \alpha = \pi \quad \mathrm{and} \quad \alpha(\widetilde{I}) = \widetilde{I} \}
\end{equation}
where we set $A(\pi)=A(\pi,[R,R])$.  Since $\widehat{F}_d$ is finitely generated as a profinite group, there are finitely many surjections $\widehat{F}_d \to H$, and $A(\pi)$ is of finite index in $\mathrm{Aut}(\widehat{F}_d)$.  In general, $A(\pi, \widetilde{I})$ need not have finite index in $\mathrm{Aut}(\widehat{F}_d)$, but it will if $\widetilde{I}$ is a characteristic subgroup of $R$.

We have an exact sequence
\begin{equation}
\label{eq:cut} 
1 \to \widetilde{R} \to \widehat{F}_d/\widetilde{I} \to H \to 1 
\end{equation} 
with some extension class $\beta \in H^2(H, \widetilde{R})$.  Since $\widetilde{R}$ is abelian and profinite, there is a well-defined continuous action of $H$ on $\widetilde{R}$.  This gives rise to a linear representation $\rho:A(\pi, \widetilde{I})\to\mathrm{GL}(\widetilde{R})$.

Let $\mathrm{Aut}_H(\widetilde{R})$ be the group of continuous automorphisms of $\widetilde{R}$ that commute with the action of $H$. Each element of $\mathrm{Aut}_H(\widetilde{R})$ induces an automorphism of $H^2(H,\widetilde{R})$.  We define 
\begin{equation}
\label{eq:OK}
\mathrm{Aut}_{H,\beta}(\widetilde{R}) = \{\gamma \in \mathrm{Aut}_H(\widetilde{R}): \gamma_*(\beta) = \beta\}.
\end{equation}
Since $\widetilde{R}$ is a quotient of a finitely generated $\widehat{\mathbb{Z}}$-module, $H^2(H,\widetilde{R})$ is finite, and $\mathrm{Aut}_{H,\beta}(\widetilde{R})$ has finite index in $ \mathrm{Aut}_H(\widetilde{R})$.

We will prove the following result, which implies Theorem \ref{thm:main2}.

\begin{theorem}
\label{thm:main} 
For every $d\ge 1$ and every $\pi$ and $\widetilde{I}$, the image $\rho(A(\pi,\widetilde{I}))$ equals $\mathrm{Aut}_{H,\beta}(\widetilde{R})$. In other words, the action of $A(\pi,\widetilde{I})$ on $\widetilde{R}$ gives a surjection
\begin{equation}
\label{eq:surjective}
A(\pi,\widetilde{I})  \to \mathrm{Aut}_{H,\beta}(\widetilde{R}).
\end{equation}
\end{theorem}

This construction is the profinite analog of the one used in \cite{LG} to construct linear representations of $\mathrm{Aut}(F_d)$. Note, though, that Theorem \ref{thm:main} shows that in the profinite case  the representation is surjective.  In the discrete case this may not be true;  see \cite[Prop. 9.7]{LG}. Here is a concrete example.

\begin{example}
\label{ex:first}
Let  $H$ be a cyclic group of prime order $p$, and let $k=\mathbb{Q}(\zeta_p)$ for a primitive $p^{\mathrm{th}}$ root of unity $\zeta_p$ in $\overline{\mathbb{Q}}$.  Suppose $\pi:F_d\to H$ is a surjection with kernel $\mathcal{R}$. In this case, $\mathbb{Q}\otimes_{\mathbb{Z}} \mathcal{R}/[\mathcal{R},\mathcal{R}]$ is isomorphic to $\mathbb{Q}[H]^{d-1}\oplus \mathbb{Q}$ as a $\mathbb{Q}[H]$-module, i.e. to $k^{d-1}\oplus \mathbb{Q}^d$ (see, for example, \cite[\S2]{Gruenberg}). Theorem 1.4 of \cite{LG} implies that if $d \ge 3$, the finite index subgroup $\mathcal{A}(\pi)$ of $\mathcal{A}=\mathrm{Aut}(F_d)$ has a ``large" image  in the $H$-module automorphism group of $k^{d-1}\oplus \mathbb{Q}^{d}$ in the following sense.  After choosing a suitable basis for $\mathbb{Q} \otimes_{\mathbb{Z}} \mathcal{R}/[\mathcal{R},\mathcal{R}] \cong k^{d-1} \oplus \mathbb{Q}^d$, the $\mathbb{Z}[H]$ submodule $\mathcal{R}/[\mathcal{R},\mathcal{R}]$ is commensurable with $\mathbb{Z}[\zeta_p]^{d-1} \oplus \mathbb{Z}^{d}$, where a chosen generator $h \in H$ acts on $\mathbb{Z}[\zeta_p]$ by multiplication by $\zeta_p$ and $h$ acts trivially on $\mathbb{Z}$.  The intersection of $\rho_\pi(\mathcal{A}(\pi))$ with  $\mathrm{GL}_{d-1}(\mathbb{Z}[\zeta_p])\times \mathrm{GL}_d(\mathbb{Z})$ has finite index in $\rho_\pi(\mathcal{A}(\pi))$, and the intersection $\rho_\pi(\mathcal{A}(\pi)) \cap \left(\mathrm{SL}_{d-1}(\mathbb{Z}[\zeta_p])\times \mathrm{SL}_d(\mathbb{Z})\right)$ is of finite index in $\mathrm{SL}_{d-1}(\mathbb{Z}[\zeta_p])\times \mathrm{SL}_d(\mathbb{Z})$. This way it is shown in \cite{LG} that $\mathrm{Aut}(F_d)$ has many arithmetic subgroups as quotients. It was left open in \cite{LG} whether the image $\rho_\pi(\mathcal{A}(\pi))$ is of finite index even in $\mathrm{GL}_{d-1}(\mathbb{Z}[\zeta_p])\times \mathrm{GL}_d(\mathbb{Z})$. This is the question ``(SL) or not (SL)?" discussed in \cite[\S8]{LG}. By using the recently proved result that $\mathrm{Aut}(F_d)$ has Kazhdan property (T) for $d\ge 4$ (see \cite{PropertyT5}, \cite{PropertyT}, \cite{PropertyT4}), so that the abelianization of each finite index subgroup is finite, it follows that (at least when $d\ge 4$), the answer to the ``(SL) or not (SL)?'' question is (SL). In other words, the image there is always commensurable with the SL group. So it is of infinite index in $\mathrm{GL}_{d-1}(\mathbb{Z}[\zeta_p])\times \mathrm{GL}_d(\mathbb{Z})$ if $p>3$.

Theorem \ref{thm:main}  gives a more precise and stronger result in the category of profinite groups.  As remarked above, $\overline{R} = R/[R,R]\cong \widehat{\mathbb{Z}}^{(d-1)p} \oplus \widehat{\mathbb{Z}}$.  Concerning the action of $H$, we can prove a more precise statement if we divide  $\overline{R}$ by its pro-$p$ Sylow subgroup $\overline{R}(p)$ (i.e. by the  $\mathbb{Z}_p$-factors). Let $\widetilde{I}$ be the inverse image of $\overline{R}(p)$ under the natural map $R \to \overline{R}$.  Define
$$\widehat{\mathbb{Z}}_{(p)} = \widehat{\mathbb{Z}}/\mathbb{Z}_p\cong \prod_{\ell\ne p} \mathbb{Z}_\ell.$$
Then $\widetilde{R} = \overline{R} /\overline{R}(p) = R/\widetilde{I}$ is isomorphic to $\widehat{\mathbb{Z}}_{(p)}^{|H|(d-1)+1}$ but this time we can say that $\widetilde{R}$ is isomorphic to $\widehat{\mathbb{Z}}_{(p)}[H]^{d-1}\oplus \widehat{\mathbb{Z}}_{(p)} \cong \widehat{\mathbb{Z}}_{(p)}[\zeta_p]^{d-1} \oplus \widehat{\mathbb{Z}}_{(p)}^{d}$ as an $H$-module since $|H|$ is a unit in $\widehat{\mathbb{Z}}_{(p)}$. Theorem \ref{thm:main} implies that $A(\pi,\widetilde{I})$, which is a finite index subgroup of $\mathrm{Aut}(\widehat{F}_d)$, is mapped \textbf{onto} $\mathrm{GL}_{d-1}(\widehat{\mathbb{Z}}_{(p)}[\zeta_p])\times  \mathrm{GL}_d(\widehat{\mathbb{Z}}_{(p)})$. (Note that $\mathrm{SL}_{d-1}(\widehat{\mathbb{Z}}_{(p)}[\zeta_p])\times  \mathrm{SL}_d(\widehat{\mathbb{Z}}_{(p)})$ is of infinite index in $\mathrm{GL}_{d-1}(\widehat{\mathbb{Z}}_{(p)}[\zeta_p])\times  \mathrm{GL}_d(\widehat{\mathbb{Z}}_{(p)})$.)
\end{example}

More examples of this kind can be worked out by considering the various examples studied in \cite{LG}. The advantage of Theorem \ref{thm:main} over \cite[Theorem 4.1]{LG} is that it gives the exact image (and not only up to a finite index subgroup as in \cite{LG}, and only under some additional assumptions on $H$ and $\pi$). It implies that the question ``(SL) or not (SL)?'' from \cite[\S8]{LG} has a clear answer here: it is never just SL. Moreover, a careful look at Example \ref{ex:first} shows that not only $\mathrm{Aut}(\widehat{F}_d)$ has an infinite abelianization, since it clearly surjects onto $\mathrm{GL}_d(\widehat{\mathbb{Z}})$ and hence onto $\widehat{\mathbb{Z}}^*$, but that also $\mathrm{Aut}^1(\widehat{F}_d)$, which is defined as the preimage of $\mathrm{SL}_d(\widehat{\mathbb{Z}})$, has a finite index subgroup with an infinite abelianization. This implies (at least for $d\ge 4$) that $\mathrm{Aut}^1(F_d)$ is not dense in $\mathrm{Aut}^1(\widehat{F}_d)$, since the former has property (T).

\medskip

To prove Theorem \ref{thm:main} we need the following well-known result of Gasch\"utz;  see \cite{Ga} and  \cite[Lemma 17.7.2]{FJ}.

\begin{lemma}
\label{lem:first}
Suppose $\psi:G_1\to G_2$ is a continuous surjective group homomorphism of profinite groups. Suppose $S_2$ is a finite set of generators of $G_2$ with $|S_2|=d$, and that the minimal number of generators of $G_1$ is less than or equal to $d$. Then there exists a set of generators $S_1$ of $G_1$ such that  $\psi(S_1)=S_2$.
\end{lemma}

This lemma has the following easy corollary (compare to \cite[\S2]{LubotzkyProfinite}).  

\begin{corollary}
\label{cor:firstsecond}
If $N$ is a closed normal subgroup of $\widehat{F}_d$, then every automorphism of $\widehat{F}_d/N$ can be lifted to an automorphism of $\widehat{F}_d$ preserving $N$. In particular, this holds for $N=R$.
\end{corollary}

The identification of second group cohomology classes with equivalence classes of group extensions shows the
following result. 

\begin{lemma}
\label{lem:second}
Assume $B$ is an abelian $H$-module and that there is a short exact sequence
\begin{equation}
\label{eq:extclass0}
\xymatrix{
1 \ar[r]& B \ar[r]& \widetilde{H} \ar[r] &H \ar[r] & 1.
}
\end{equation}
Let $\widetilde{\beta} \in H^2(H,B)$ be the extension class of this sequence.  Let $\gamma:B \to B$ be an $H$-module automorphism.  Then  $\gamma$  fixes $\widetilde{\beta}$ if and only if there is a diagram of automorphisms
\begin{equation}
\label{eq:extclass}
\xymatrix{
1 \ar[r]& B \ar[r]\ar[d]_{\gamma}& \widetilde{H} \ar[r] \ar[d]^{\widetilde{\gamma}}&H \ar[r]\ar@{=}[d] & 1\\
1 \ar[r]& B \ar[r]& \widetilde{H} \ar[r] &H \ar[r] & 1
}
\end{equation}
in which the right vertical isomorphism is the identity map on $H$.
\end{lemma}
 
\begin{proof}[Proof of Theorem \ref{thm:main}]  
An element of $A(\pi,\widetilde{I})  = \{\alpha \in \mathrm{Aut}(\widehat{F}_d): \pi \circ \alpha = \pi \mbox{ and } \alpha(\widetilde{I}) = \widetilde{I} \}$ induces an automorphism of the sequence (\ref{eq:cut}) which is the identity on $H$.  Therefore, Lemma \ref{lem:second} implies that the restriction of this element to $\widetilde{R}$ lies in $\mathrm{Aut}_{H,\beta}(\widetilde{R})$.  We now suppose that $\gamma$ is an element of $\mathrm{Aut}_{H,\beta}(\widetilde{R})$.  By Lemma \ref{lem:second}, there is an automorphism $\widetilde{\gamma}$ of $\widehat{F}_d/\widetilde{I}$ that fits in a diagram of automorphisms
\begin{equation}
\label{eq:extclassnow}
\xymatrix{
1 \ar[r]& \widetilde{R} \ar[r]\ar[d]_{\gamma}&\widehat{F}_d/\widetilde{I} \ar[r] \ar[d]^{\widetilde{\gamma}}&H \ar[r]\ar@{=}[d] & 1\;\;\\
1 \ar[r]& \widetilde{R} \ar[r]&\widehat{F}_d/\widetilde{I} \ar[r] &H \ar[r] & 1.
}
\end{equation}
By  Corollary \ref{cor:firstsecond}, we can lift $\widetilde{\gamma}$ to an automorphism $\alpha$ of $\widehat{F}_d$. This lift will lie in $A(\pi,\widetilde{I})$, which shows that the homomorphism in Theorem \ref{thm:main} is surjective.
\end{proof}

\section{Bely\u{\i}'s construction of an injection $\mathrm{Gal}(\overline{\mathbb{Q}}/\mathbb{Q}) \to \mathrm{Aut}(\widehat{F}_2)$.}
\label{s:GT}

In this section, we review a construction of Bely\u{\i} in  \cite{Belyi1, Belyi, HarbaterSchneps} of an injection $\iota: \mathrm{Gal}(\overline{\mathbb{Q}}/\mathbb{Q}) \to \mathrm{Aut}(\widehat{F}_2)$ which is essential to our work.  An explicit description of this construction is needed, in particular, to show that on passing to a finite index subgroup, it is compatible with the actions of $\mathrm{Gal}(\overline{\mathbb{Q}}/\mathbb{Q})$ and $\mathrm{Aut}(\widehat{F}_2)$ on the generalized Jacobians of curves.

We will identify $\mathbb{Q}(\mathbb{P}^1_{\mathbb{Q}})$ with $\mathbb{Q}(t)$ for a fixed choice of affine parameter $t$.  Then $\overline{\mathbb{Q}}(\mathbb{P}^1_{\overline{\mathbb{Q}}})$ is identified with  $\overline{\mathbb{Q}}(t)$. Let $\overline{\mathbb{Q}}(t)^{\mathrm{cl}}$ be a fixed algebraic closure of $\overline{\mathbb{Q}}(\mathbb{P}^1_{\overline{\mathbb{Q}}}) = \overline{\mathbb{Q}}(t)$.  We have an exact sequence of \'etale fundamental groups 
\begin{equation}
\label{eq:fundamental}
1 \to \pi_1(\mathbb{P}^1_{\overline{\mathbb{Q}}} - \{0,1,\infty\},\eta) \to \pi_1(\mathbb{P}^1_{\mathbb{Q}} - \{0,1,\infty\},\eta) \to \mathrm{Gal}(\overline{\mathbb{Q}}/\mathbb{Q}) \to 1
\end{equation}
when $\eta:\mathrm{Spec}(\overline{\mathbb{Q}}(t)^{\mathrm{cl}}) \to \mathbb{P}^1_{\overline{\mathbb{Q}}} - \{0,1,\infty\}$ is the geometric point determined by $\overline{\mathbb{Q}}(t)^{\mathrm{cl}}$.  This gives a homomorphism $\mathrm{Gal}(\overline{\mathbb{Q}}/\mathbb{Q}) \to \mathrm{Out}(\pi_1(\mathbb{P}^1_{\overline{\mathbb{Q}}} - \{0,1,\infty\},\eta))$.  Bely\u{\i} showed in \cite{Belyi1,Belyi} how one can lift this homomorphism to an injection 
\begin{equation}
\label{eq:inject}
\iota: \mathrm{Gal}(\overline{\mathbb{Q}}/\mathbb{Q}) \to \mathrm{Aut}(\pi_1(\mathbb{P}^1_{\overline{\mathbb{Q}}} - \{0,1,\infty\},\eta))
\end{equation}
in the following way.

Let $L$ be the maximal extension of $\overline{\mathbb{Q}}(t)$ contained in $\overline{\mathbb{Q}}(t)^{\mathrm{cl}}$ that is unramified outside $\{0,1,\infty\}$.  Define 
 $$\Gamma_{\overline{\mathbb{Q}}} = \mathrm{Gal}(L/\overline{\mathbb{Q}}(t)) = \pi_1(\mathbb{P}^1_{\overline{\mathbb{Q}}} - \{0,1,\infty\},\eta).$$ Let $x$ and $y$ be topological generators of decomposition groups in $\Gamma_{\overline{\mathbb{Q}}}$ over $0$ and $1$.  These decomposition groups correspond to projective systems $\{P_{i,0}\}_i$ and $\{P_{i,1}\}_i$ of discrete valuations over $0$ and $1$, respectively, in a countable cofinal system indexed by $i$ of finite extensions of $\overline{\mathbb{Q}}(t)$ in $L$.  The topological fundamental group of $\mathbb{P}^1(\mathbb{C}) - \{0,1,\infty\}$ is the free group generated by loops around $0$ and $1$.  It follows that the group $\Gamma_{\overline{\mathbb{Q}}} $ is the profinite completion $\widehat {\langle x,y\rangle }$ of the free group generated by $x$ and $y$.

For  $\sigma \in \mathrm{Gal}(\overline{\mathbb{Q}}/\mathbb{Q})$, we lift $\sigma$ canonically to an element of $\mathrm{Gal}(\overline{\mathbb{Q}}(t)/\mathbb{Q}(t))$ by letting $\sigma$ fix $t$.  Choose a further  lift $\widetilde{\sigma}$ of $\sigma$ to $\mathrm{Aut}(L/\mathbb{Q}(t)) \subset \mathrm{Aut}(L/\mathbb{Q})$.  There is a transitive  action of $\Gamma_{\overline{\mathbb{Q}}}$ on the discrete valuations of any finite Galois extension of $\overline{\mathbb{Q}}(t)$ in $L$ that lie over the discrete valuation of $\overline{\mathbb{Q}}(t)$ associated to the point $0$.  It follows that there is an element $\gamma \in \mathrm{Gal}(L/\overline{\mathbb{Q}}(t))$ such that $\gamma(\{P_{i,0}\}_i) = \widetilde{\sigma}(\{P_{i,0}\}_i)$.  Then conjugation by $\gamma^{-1} \widetilde{\sigma} $ takes 
the profinite group $\widehat{\langle x \rangle}$ generated by $x$ back to itself.  

Define 
$$\Gamma_{\mathbb{Q}} = \mathrm{Gal}(L/\mathbb{Q}(t))$$
and let $[\Gamma_{\overline{\mathbb{Q}}}, \Gamma_{\overline{\mathbb{Q}}} ]$ be the commutator subgroup of $\Gamma_{\overline{\mathbb{Q}}} \subset \Gamma_{\mathbb{Q}} $.  When we set 
$$\widetilde{\Gamma}_{\mathbb{Q}} = \frac{\Gamma_{\mathbb{Q}}}{[\Gamma_{\overline{\mathbb{Q}}}, \Gamma_{\overline{\mathbb{Q}}} ]}$$
we have an exact sequence
\begin{equation}
\label{eq:exactnice}
1 \to \Gamma_{\overline{\mathbb{Q}}}^{\mathrm{ab}} \to \widetilde{\Gamma}_{\mathbb{Q}} \to \mathrm{Gal}(\overline{\mathbb{Q}}/\mathbb{Q}) \to 1
\end{equation}
where
$$\Gamma_{\overline{\mathbb{Q}}}^{\mathrm{ab}} = \Gamma_{\overline{\mathbb{Q}}}/[\Gamma_{\overline{\mathbb{Q}}}, \Gamma_{\overline{\mathbb{Q}}} ] = \widehat{\langle x,y \rangle}^{\mathrm{ab}} \cong \widehat{\langle x \rangle } \times \widehat{ \langle y \rangle}$$
and we have identified $\mathrm{Gal}(\overline{\mathbb{Q}}(t))/\mathbb{Q}(t))$ with $\mathrm{Gal}(\overline{\mathbb{Q}}/\mathbb{Q})$.

Let $D(y)$ be the subgroup of $\Gamma_{\mathbb{Q}}$ which fixes the projective system $\{P_{i,1}\}_i$ of points lying over $1$. Since the residue field of the point $1$ on $\mathbb{P}^1_{\mathbb{Q}}$ is $\mathbb{Q}$, we have an exact sequence
$$1 \to \widehat {\langle y \rangle} \to D(y)  \to \mathrm{Gal}(\overline{\mathbb{Q}}/\mathbb{Q}) \to 1.$$
We have an injective homomorphism  $\widehat {\langle y \rangle} \to \Gamma_{\overline{\mathbb{Q}}}^{\mathrm{ab}}$ induced by the inclusion of $\widehat {\langle y \rangle}$ into $\Gamma_{\overline{\mathbb{Q}}}$.  It follows that the homomorphism $D(y) \to \widetilde{\Gamma}_{\mathbb{Q}}$ induced by $D(y) \subset \Gamma_{\mathbb{Q}}$ is injective;  let $\overline{D}(y)$ be the image of this homomorphism.  For simplicity we will identify $\widehat{\langle x \rangle}$  and $\widehat{\langle y \rangle}$ with their images in $\Gamma_{\overline{\mathbb{Q}}}^{\mathrm{ab}}$.
 
We now claim that $\widehat{\langle x \rangle}$ is a set of representatives for the left cosets $z \overline{D}(y)$ of $\overline{D}(y)$ in $\widetilde{\Gamma}_{\mathbb{Q}}$.  First, these cosets are disjoint, since if $z \overline{D}(y) = z' \overline{D}(y)$ for some $z,z' \in \widehat{\langle x \rangle}$ then $z^{-1} z'$ is in the intersection of $\overline{D}(y)$ with the kernel of $\widetilde{\Gamma}_{\mathbb{Q}} \to \mathrm{Gal}(\overline{\mathbb{Q}}/\mathbb{Q})$.   This intersection is
$$\overline{D}(y) \cap \Gamma_{\overline{\mathbb{Q}}}^{\mathrm{ab}} = \widehat{\langle y \rangle} \subset \Gamma_{\overline{\mathbb{Q}}}^{\mathrm{ab}} = \widehat{\langle x \rangle } \times \widehat{ \langle y \rangle}.$$
The only element of $\widehat{\langle x \rangle}$ lying in this intersection is the identity element, so $z = z'$.  We must now show that every element $\psi$ of $\widetilde{\Gamma}_{\mathbb{Q}}$ lies in some coset $z \overline{D}(y)$ with $z \in \widehat{\langle x \rangle}$. To show this, first note that we can find $d \in \overline{D}(y)$ with the same image under $\widetilde{\Gamma}_{\mathbb{Q}} \to \mathrm{Gal}(\overline{\mathbb{Q}}/\mathbb{Q})$ as $\psi$.  Then $\psi d^{-1} \in \Gamma_{\overline{\mathbb{Q}}}^{\mathrm{ab}} = \widehat{\langle x \rangle } \times \widehat{ \langle y \rangle}$, so $\psi d^{-1} \in z \widehat{ \langle y \rangle}$ for some $z \in \widehat{\langle x \rangle }$. Hence, $\psi \in z \widehat{ \langle y \rangle} d \subset z \overline{D}(y)$ and we are done.

We now return to the element $\gamma^{-1} \widetilde{\sigma} $ of $\Gamma_{\mathbb{Q}}$. The image of this element in $\widetilde{\Gamma}_{\mathbb{Q}}$ lies in a coset $\nu \overline{D}(y)$ for a unique $\nu \in \widehat{\langle x \rangle}$.  Therefore, $\nu^{-1} \gamma^{-1} \widetilde{\sigma} \in \Gamma_{\mathbb{Q}}$ has the same image in $\widetilde{\Gamma}_{\mathbb{Q}}$ as an element $\mu$ of $D(y)$, and $\mu$ is unique since $\nu$ was unique and the map $D(y) \to \overline{D}(y)$ is an isomorphism.  Hence, $\nu^{-1} \gamma^{-1} \widetilde{\sigma} = \lambda \cdot \mu$ for a unique $\lambda$ in the commutator subgroup $[ \Gamma_{\overline{\mathbb{Q}}}, \Gamma_{\overline{\mathbb{Q}}}]$ of  $\Gamma_{\overline{\mathbb{Q}}}$.  Since $\nu$ was uniquely determined by $\gamma^{-1} \widetilde{\sigma}$, both $\lambda$ and $\mu$ are as well.   The intersection of $[ \Gamma_{\overline{\mathbb{Q}}}, \Gamma_{\overline{\mathbb{Q}}}]$ with $D(y)$ is trivial.  It follows that $\sigma' = \nu^{-1} \gamma^{-1} \widetilde{\sigma} = \lambda \cdot \mu$ is a lift of $\sigma$ to $\mathrm{Gal}(L/\mathbb{Q}(t)) = \Gamma_{\mathbb{Q}}$ that has these properties:
\begin{enumerate}
\item[i.] Conjugation by $\sigma'$ takes $\widehat{\langle x \rangle}$ to itself, so $\sigma' \cdot x \cdot (\sigma')^{-1} = x^{a}$ for some $a \in \widehat{\mathbb{Z}}^*$.
\item[ii.]  Conjugation by $\sigma'$ takes $\widehat{\langle y \rangle}$ to a conjugate $\lambda \widehat{\langle y \rangle} \lambda^{-1} = (\lambda \mu) \widehat{\langle y \rangle} (\lambda \mu)^{-1} $ of $\widehat{\langle y \rangle}$ by an element $\lambda$ of the commutator subgroup $[ \Gamma_{\overline{\mathbb{Q}}}, \Gamma_{\overline{\mathbb{Q}}}]$. Thus $\sigma' \cdot y \cdot (\sigma')^{-1} = (\lambda\cdot y \cdot \lambda^{-1})^b$ for some $b \in \widehat{\mathbb{Z}}^*$.
\end{enumerate}
By considering the action of $\sigma'$ by conjugation on $\Gamma_{\overline{\mathbb{Q}}}^{\mathrm{ab}}$, one sees by Kummer theory that $a = b  = \chi_{\mathrm{cyc}}(\sigma)$ where $\chi_{\mathrm{cyc}}$ is the standard cyclotomic character.  Bely\u{\i} shows that  $\widehat{\langle x \rangle}$ and $\widehat{\langle y \rangle}$ are their own centralizers in $\Gamma_{\overline{\mathbb{Q}}}$, from which it follows that $\sigma'$ is uniquely determined by $\sigma$.    Bely\u{\i} denotes $\lambda$ by $f_\sigma$, and the Bely\u{\i} lift
$$\iota:\mathrm{Gal}(\overline{\mathbb{Q}}/\mathbb{Q}) \to \mathrm{Aut}(\Gamma_{\overline{\mathbb{Q}}}) = \mathrm{Aut}(\widehat{\langle x,y \rangle})$$
is defined by letting $\iota(\sigma)$ be conjugation by $\sigma'$.

For later use, we will describe one consequence of this construction.  Suppose $F$ is a number field and that $Y \to  \mathbb{P}^1_{\overline{\mathbb{Q}}}$ is an irreducible finite Galois $H$-cover of smooth projective irreducible curves that is unramified outside $\{0, 1, \infty\}$. Suppose $Y$ and the action of $H$ on $Y$ are defined over $F$, and that $F$ is algebraically closed in the function field $F(Y)$. We will furthermore require that the points over $0$, $1$ and $\infty$ on $Y$ are defined over $F$. Recall that $\eta$ corresponds to the choice of an algebraic closure $\overline{\mathbb{Q}}(t)^{\mathrm{cl}}$ of $\overline{\mathbb{Q}}(t)$.  We fix an embedding of $F(Y)$ into $\overline{\mathbb{Q}}(t)^{\mathrm{cl}}$.   

\begin{lemma}
\label{lem:extendit}  We can find a finite extension $F^\dagger$ of $F$  so that if $\sigma \in \mathrm{Gal}(\overline{\mathbb{Q}}/F^\dagger)$ then the element $\sigma'$ in the Bely\u{\i} construction of $\iota(\sigma)$ lies in $\mathrm{Gal}(L/F^\dagger(Y))$.
\end{lemma}

\begin{proof}
If $\sigma \in \mathrm{Gal}(\overline{\mathbb{Q}}/\mathbb{Q})$ lies in the finite index subgroup $\mathrm{Gal}(\overline{\mathbb{Q}}/F)$, we can extend $\sigma$ in a unique way to an element of $\mathrm{Gal}(\overline{\mathbb{Q}}(Y)/F(Y))$.  We can thus choose the first lift $\widetilde{\sigma} \in \mathrm{Aut}(L/\mathbb{Q})$ in Bely\u{\i}'s construction so that $\widetilde{\sigma} \in \mathrm{Gal}(L/F(Y))$.   Recall that $\{P_{i,0}\}_i$ was a cofinal system of points over $0$ in a cofinal system of finite covers of  $\mathbb{P}^1_{\overline{\mathbb{Q}}}$ that are unramified outside $0, 1$ and $\infty$.  We can assume that $Y$ is one of the covers in this system, and that all the covers in the system are in fact covers of $Y$. Let $0_Y$ be the point of $Y$ lying over $0$ which appears in the system $\{P_{i,0}\}_i$.  Since $\widetilde{\sigma}$ fixes $Y$ and $0_Y$, we see that $\{\widetilde{\sigma}(P_{i,0})\}_i$ is a system of points over $0$ which all lie over $0_Y$. It follows that the first element $\gamma \in \mathrm{Gal}(L/\overline{\mathbb{Q}}(t))$ in Bely\u{\i}'s construction can be chosen to lie in $\mathrm{Gal}(L/\overline{\mathbb{Q}}(Y))$.  However, we cannot say that the next element $\nu$ in Bely\u{\i}'s construction lies in $\mathrm{Gal}(L/\overline{\mathbb{Q}}(Y))$.  We now indicate some further hypotheses that will force this to be the case. 

Let $s$ be the exponent of the finite group $H$, i.e. the smallest positive integer such that every element of $H$ has order dividing $s$.  There is an abelian cover $U$ of $\mathbb{P}^1_{\overline{\mathbb{Q}}}$ that is unramified outside $\{0,1,\infty\}$ and whose Galois group over $\mathbb{P}^1_{\overline{\mathbb{Q}}}$ is isomorphic to the group $\mathbb{Z}/s \times \mathbb{Z}/s$.  Any automorphism in $\mathrm{Gal}(L/\overline{\mathbb{Q}}(U))$ then has image in $\mathrm{Gal}(L/\overline{\mathbb{Q}}(t))^{\mathrm{ab}} = \widehat{\langle x,y\rangle}^{\mathrm{ab}} \cong \widehat{\langle x \rangle} \times \widehat{\langle y\rangle} $ which is an $s^{\mathrm{th}}$ power. Let $F^\dagger$ be a finite extension of $F$ with the following properties.  The composite of $\overline{\mathbb{Q}}(U)$ and $\overline{\mathbb{Q}}(Y)$ in $L$ is the function field of a smooth projective irreducible curve $Y^\dagger$ that  is defined over $F^\dagger$. The action of $H^\dagger = \mathrm{Gal}(\overline{\mathbb{Q}}(Y^\dagger)/\overline{\mathbb{Q}}(t))$ is defined over $F^\dagger$.  The points of $Y^\dagger$ over $\{0,1,\infty\}$ are defined over $F^\dagger$, and $F^\dagger$ is the constant field of $F^\dagger(Y^\dagger)$.  Then $Y^\dagger$ is a cover of $Y$ and $H$ is a quotient of $H^\dagger$.  Running Bely\u{\i}'s construction now with $Y$ replaced by $Y^\dagger$ and $F$ replaced by $F^\dagger$, we arrive at a lift $\widetilde{\sigma}  \in \mathrm{Gal}(L/F^\dagger(Y^\dagger))$ and an element $\gamma \in \mathrm{Gal}(L/\overline{\mathbb{Q}}(Y^\dagger))$.   

We now recall that $\nu$ is the unique element of $\widehat{\langle x \rangle}$ such that the image of $\gamma^{-1} \widetilde{\sigma}$ in $\widetilde{\Gamma}_{\mathbb{Q}}  = \Gamma_{\mathbb{Q}}/[\Gamma_{\overline{\mathbb{Q}}}, \Gamma_{\overline{\mathbb{Q}}} ]$ lies in the image in $\widetilde{\Gamma}_{\mathbb{Q}}$ of the coset $\nu D(y)$ of the decomposition group $D(y)$ of the inverse system of points $\{P_{i,1}\}_i$ over $1$.  We can assume that the inverse system of covers used to define $\{P_{i,1}\}_i$ includes $U$ and $Y$ and $Y^{\dagger}$. Since we have arranged that $\gamma^{-1} \widetilde{\sigma}$ fixes $F^\dagger(Y^{\dagger})$ and that the points over $\{0,1,\infty\}$ on $Y^\dagger$ are defined over $F^\dagger$, the action of $\gamma^{-1} \widetilde{\sigma}$ fixes the point $P_{U,1}$ of the cover $U$ over $1$ in the above inverse system.  Now, $U$ is an abelian cover of $\mathbb{P}^1_{\overline{\mathbb{Q}}}$ and  $\gamma^{-1} \widetilde{\sigma}$ has the same image in $\widetilde{\Gamma}_{\mathbb{Q}}$ as $\nu z$ for some $z$ in $D(y)$.  Since $D(y)$ fixes the inverse system $\{P_{i,1}\}_i$, we conclude that $\nu \in \widehat{\langle x \rangle}$ must fix $P_{U,1}$ because $\gamma^{-1} \widetilde{\sigma}$  does.  Therefore, the image of $\nu \in \widehat{\langle x \rangle} \subset \Gamma_{\overline{\mathbb{Q}}} $ in $\mathrm{Gal}(\overline{\mathbb{Q}}(U)/\overline{\mathbb{Q}}(t)) = \mathbb{Z}/s \times \mathbb{Z}/s$ is in the inertia group of a point $P_{U,0}$ over $0$ in $U$ as well as  in the inertia group of the point $P_{U,1}$.  These inertia subgroups of $\mathrm{Gal}(\overline{\mathbb{Q}}(U)/\overline{\mathbb{Q}}(t))$ have trivial intersection, so $\nu \in \widehat{\langle x \rangle}$ has trivial image in $\mathrm{Gal}(\overline{\mathbb{Q}}(U)/\overline{\mathbb{Q}}(t))$.  This forces $\nu$ to be the $s^{\mathrm{th}}$ power of an element of $\widehat{\langle x \rangle}$.  Since $s$ was the exponent of $H$, we conclude that $\nu$ has trivial image in $H = \mathrm{Gal}(\overline{\mathbb{Q}}(Y)/\overline{\mathbb{Q}}(t))$. Therefore, $\sigma' = \nu^{-1} \gamma^{-1} \widetilde{\sigma}$ lies in $\mathrm{Gal}(L/F^\dagger(Y))$ and we are done.
\end{proof}

\section{Galois representations}
\label{s:GalReps}

Let $X$ be a smooth projective irreducible curve defined over $\overline{\mathbb{Q}}$.  By Bely\u{\i}'s Theorem \cite[Theorem 4]{Belyi1}, there is a non-constant morphism $\lambda:X \to \mathbb{P}^1_{\overline{\mathbb{Q}}}$ which is unramified outside $\{0,1,\infty\}$. Let $H$ be the Galois group of the Galois closure $\widetilde{\lambda}:Y  \to \mathbb{P}^1_{\overline{\mathbb{Q}}}$ of $\lambda$. We will identify $\widehat{F}_2$ with the \'etale fundamental group $\pi_1(\mathbb{P}^1_{\overline{\mathbb{Q}}} - \{0,1,\infty\},\eta)$ appearing in (\ref{eq:inject}) by the choices described in the previous section. We will also view $G_{\mathbb{Q}}=\mathrm{Gal}(\overline{\mathbb{Q}}/\mathbb{Q})$ as a subgroup of $\mathrm{Aut}(\widehat{F}_2)$ via the injection $\iota$ of (\ref{eq:inject}).  Let $\pi:\widehat{F}_2 \to H$ be the homomorphism associated to $\widetilde{\lambda}$ and the choice of a geometric point of $Y$ over $\eta$.  Fix $d = 2$ and let 
$$\rho: A(\pi) \to \mathrm{Aut}_{H,\beta}(\overline{R})$$
be the surjective homomorphism defined in Theorem \ref{thm:main2}.

We will prove the following result, which implies Theorem \ref{thm:main1}.

\begin{theorem}
\label{thm:restriction}
The group $\overline{R}$ is isomorphic to the Galois group of the maximal abelian cover of $Y$ that is unramified outside the ramification locus $\widetilde{\lambda}^{-1}(\{0,1,\infty\})$.   This Galois group is isomorphic to the adelic Tate module $T_{\widetilde{\lambda}}(Y)$ of the generalized Jacobian $J_{\widetilde{\lambda}}(Y)$ of $Y$ with respect to $\widetilde{\lambda}^{-1}(\{0,1,\infty\})$.  Let $\mathcal{J}\subset H$ be the Galois group of $Y$ over $X$.  The coinvariants $\overline{R}_\mathcal{J}$ of $\overline{R}$ with respect to $\mathcal{J}$ have finite torsion, and the quotient $\overline{R}_\mathcal{J}^{\mathrm{cotor}}$ of $\overline{R}_\mathcal{J}$ by this torsion subgroup is naturally identified with a subgroup of finite index in the Galois group of the maximal abelian cover of $X$ that is unramified outside the ramification locus $\lambda^{-1}(\{0,1,\infty\})$.   The latter Galois group is $T_{\lambda}(X)$.  There is an $H$-equivariant action of $A(\pi)$ on $T_{\widetilde{\lambda}}(Y)$, and there is a finite index subgroup $A_X(\pi)$ of $A(\pi)$ that acts on $T_\lambda(X)$.

There is a finite extension $F$ of $\mathbb{Q}$ over which $X$, $Y$,  $J_{\lambda}(X)$, $J_{\widetilde{\lambda}}(Y)$ and the action of $H$ on $Y$ are defined. There is a natural action of $G_F=\mathrm{Gal}(\overline{\mathbb{Q}}/F)$ on $T_{\widetilde{\lambda}}(Y)$ and $T_{\lambda}(X)$.  There are finite index normal subgroups $A_Y$ of $A(\pi)$ and $A_X$ of $A_X(\pi)$ such that the action of $G_F \cap \iota^{-1}(A_Y)$ (resp. $G_F \cap \iota^{-1}(A_X)$) on $T_{\widetilde{\lambda}}(Y)$ (resp. $T_{\lambda}(X)$)  agrees with the action of $A_Y$ (resp. $A_X$) under  Bely\u{\i}'s embedding $G_F \xrightarrow{\iota} \mathrm{Aut}(\widehat{F}_2)$.  
\end{theorem}

\begin{proof} 
The identification of $\overline{R}$ with $T_{\widetilde{\lambda}}(Y)$ is shown by Serre in \cite[\S I.2]{GroupesAlgebriques}. Let $L$ be the maximal abelian extension of $\overline{\mathbb{Q}}(Y)$ that is Galois over $\overline{\mathbb{Q}}(X)$, unramified outside $\widetilde{\lambda}^{-1}(\{0,1,\infty\})$ and for which $\mathrm{Gal}(L/\overline{\mathbb{Q}}(Y))$ is central in $\mathrm{Gal}(L/\overline{\mathbb{Q}}(X))$.  Then $L$ contains the maximal abelian extension $L'$ of $\overline{\mathbb{Q}}(X)$ that is unramified outside $\lambda^{-1}(\{0,1,\infty\})$. Furthermore, $\mathrm{Gal}(L/\overline{\mathbb{Q}}(Y)) = T_{\widetilde{\lambda}}(Y)_\mathcal{J}$ and there is a central extension of groups
$$1 \to T_{\widetilde{\lambda}}(Y)_\mathcal{J} \to \mathrm{Gal}(L/\overline{\mathbb{Q}}(X)) \to \mathcal{J} \to 1.$$
Letting these groups act trivially on $\mathbb{Q}/\mathbb{Z}$, the Hochschild-Serre spectral sequence gives an exact sequence of low degree terms
$$0 \to H^1(\mathcal{J},\mathbb{Q}/\mathbb{Z}) \to H^1(\mathrm{Gal}(L/\overline{\mathbb{Q}}(X)),\mathbb{Q}/\mathbb{Z}) \to H^1(T_{\widetilde{\lambda}}(Y)_\mathcal{J},\mathbb{Q}/\mathbb{Z}) \to H^2(\mathcal{J},\mathbb{Q}/\mathbb{Z}).$$
Here $H^i(\mathcal{J},\mathbb{Q}/\mathbb{Z})$ is finite for $i 	\ge 1$.  The maximal abelian quotient of $\mathrm{Gal}(L/\overline{\mathbb{Q}}(X))$ is $\mathrm{Gal}(L'/\overline{\mathbb{Q}}(X)) = T_\lambda(X)$ and
$$H^1(\mathrm{Gal}(L/\overline{\mathbb{Q}}(X)),\mathbb{Q}/\mathbb{Z})  = \mathrm{Hom}(\mathrm{Gal}(L'/\overline{\mathbb{Q}}(X)), \mathbb{Q}/\mathbb{Z}).$$
Moreover,
$$H^1(T_{\widetilde{\lambda}}(Y)_\mathcal{J},\mathbb{Q}/\mathbb{Z}) = \mathrm{Hom}(T_{\widetilde{\lambda}}(Y)_\mathcal{J},\mathbb{Q}/\mathbb{Z}).$$
Since $T_{\widetilde{\lambda}}(Y)_\mathcal{J}$ is abelian, we obtain that  $T_{\widetilde{\lambda}}(Y)_\mathcal{J}$ maps with finite kernel and cokernel to  $\mathrm{Gal}(L'/\overline{\mathbb{Q}}(X)) = T_\lambda(X)$.  Because $T_\lambda(X)$ is torsion free, we conclude that $T_{\widetilde{\lambda}}(Y)_\mathcal{J}$ has finite torsion, and the quotient $T_{\widetilde{\lambda}}(Y)_\mathcal{J}^{\mathrm{cotor}}$ of $T_{\widetilde{\lambda}}(Y)_\mathcal{J}$ by its torsion subgroup is isomorphic to a subgroup of finite index in $T_\lambda(X)$.  

We now identify $T_{\widetilde{\lambda}}(Y)$ with $\overline{R}$ and apply Theorem \ref{thm:main2}.  This shows that there is an $H$-equivariant action of $A(\pi)$ on $T_{\widetilde{\lambda}}(Y)$. In particular, this action  descends to an action of $A(\pi)$ on $T_{\widetilde{\lambda}}(Y)_\mathcal{J}^{\mathrm{cotor}}$.  Let $m$ be the index of $T_{\widetilde{\lambda}}(Y)_\mathcal{J}^{\mathrm{cotor}}$ in $ T_\lambda(X)$.   We have a sequence of inclusions
$$m\, T_{\widetilde{\lambda}}(Y)_\mathcal{J}^{\mathrm{cotor}}\subset m\, T_\lambda(X) \subset T_{\widetilde{\lambda}}(Y)_\mathcal{J}^{\mathrm{cotor}}.$$
Since $A(\pi)$ acts on the finite group $T_{\widetilde{\lambda}}(Y)_\mathcal{J}^{\mathrm{cotor}}/m\, T_{\widetilde{\lambda}}(Y)_\mathcal{J}^{\mathrm{cotor}}$, there exists a finite index subgroup $A_X(\pi)$ of $A(\pi)$ such that $A_X(\pi)$ acts trivially on $T_{\widetilde{\lambda}}(Y)_\mathcal{J}^{\mathrm{cotor}}/m\, T_{\widetilde{\lambda}}(Y)_\mathcal{J}^{\mathrm{cotor}}$. Therefore, $A_X(\pi)$ preserves $m\,T_\lambda(X)$ and hence $A_X(\pi)$ acts on $T_\lambda(X)$.

It remains to show that we can  shrink $A(\pi)$ and $A_X(\pi)$ to smaller finite index normal subgroups $A_Y$ and $A_X$, respectively, if necessary, so that the action of $G_F \cap \iota^{-1}(A_Y)$ (resp. $G_F \cap \iota^{-1}(A_X)$) on $T_{\widetilde{\lambda}}(Y)$ (resp. $T_{\lambda}(X)$)  agrees with the action of $A_Y$ and $A_X$ under  Bely\u{\i}'s embedding $\iota:G_F \to \mathrm{Aut}(\widehat{F}_2)$.  We will use the description of $\iota$ given in \S \ref{s:GT}. 

In the arguments below, we will need to enlarge $F$ to a finite extension $F^\dagger$ of $F$.  Since $\iota:G_F \to \mathrm{Aut}(\widehat{F}_2)$ is a continuous injective homomorphism of profinite groups, it follows from \cite[Prop. 2.1.5]{Zalesskii} that $\iota$ is a homeomorphism onto its image when the image is given the topology induced by that of $\mathrm{Aut}(\widehat{F}_2)$.  Therefore, there are finite index normal subgroups of $A(\pi)$ and $A_X(\pi)$ with the property that their intersection with $\iota(G_F)$ is contained in $\iota(G_{F^\dagger})$. Thus in what follows we are free to enlarge $F$ by a finite extension in order to prove the existence of finite index subgroups $A_Y$ and $A_X$ with the desired properties.
 
By Lemma \ref{lem:extendit}, we can now replace $F$ by a larger finite extension $F^\dagger$ to be able to assume that if $\sigma \in G_F$ then $\iota(\sigma)$ is conjugation by an element $\sigma' \in \mathrm{Gal}(L/F(Y))$ which restricts to $\sigma$ on $\overline{\mathbb{Q}} \subset L$. We have a tower of fields $F(Y) \subset \overline{\mathbb{Q}}(Y) \subset \overline{\mathbb{Q}}(Y)^{\mathrm{ab}}\subset L$ in which $\overline{\mathbb{Q}}(Y)^{\mathrm{ab}}$ is the largest abelian extension of $\overline{\mathbb{Q}}(Y)$ inside $L$.  By   \cite[\S I.2]{GroupesAlgebriques}, $\mathrm{Gal}(\overline{\mathbb{Q}}(Y)^{\mathrm{ab}}/\overline{\mathbb{Q}}(Y))$ is naturally isomorphic to the adelic Tate module $T_{\widetilde{\lambda}}(Y)$, and the action of $\sigma$ on $T_{\widetilde{\lambda}}(Y)$ corresponds to the conjugation action of any lift $\widehat{\sigma}$ of $\sigma$ to an element of $\mathrm{Gal}(\overline{\mathbb{Q}}(Y)^{\mathrm{ab}}/F(Y))$.  Note that this conjugation action does not depend on the choice of the lift $\widehat{\sigma}$ because $\mathrm{Gal}(\overline{\mathbb{Q}}(Y)^{\mathrm{ab}}/\overline{\mathbb{Q}}(Y))$ is an abelian normal subgroup of  $\mathrm{Gal}(\overline{\mathbb{Q}}(Y)^{\mathrm{ab}}/F(Y))$.  However, we know that the action of $\iota(\sigma)$ on $\mathrm{Gal}(\overline{\mathbb{Q}}(Y)^{\mathrm{ab}}/\overline{\mathbb{Q}}(Y))$ is via the conjugation action by $\sigma' \in \mathrm{Gal}(L/F(Y))$, so the action of $\iota(\sigma)$ on $\mathrm{Gal}(\overline{\mathbb{Q}}(Y)^{\mathrm{ab}}/\overline{\mathbb{Q}}(Y))$ agrees with the conjugation action of $\widehat{\sigma}$.  This shows that the action of $\iota(\sigma)$ on $T_{\widetilde{\lambda}}(Y)$ agrees with the natural action of $\sigma$. 
\end{proof}

\begin{remark}
\label{rem:ell}  
The statements  in Theorem \ref{thm:restriction} hold if one replaces $T_{\widetilde{\lambda}}(Y)$  and $T_{\lambda}(X)$ by their maximal pro-$\ell$ quotients $T_{\ell,\widetilde{\lambda}}(Y)$ and $T_{\ell,\lambda}(X)$, the latter being the $\ell$-adic Tate modules of the generalized Jacobians of $Y$ and $X$ with respect to $\widetilde{\lambda}^{-1}(\{0,1,\infty\})$ and $\lambda^{-1}(\{0,1,\infty\})$, respectively.
\end{remark}

We obtain the following consequence of Theorem \ref{thm:restriction}.

\begin{corollary}
\label{cor:Rtilde}
Assume the notation from Theorem \ref{thm:restriction}. Suppose $\widetilde{I}$ is a (closed) normal  subgroup of $\widehat{F}_2$ such that $[R,R]\subset \widetilde{I} \subset R$, and let $\widetilde{R}=R/\widetilde{I}$. Then $\widetilde{R}$ is a quotient module of $\overline{R}=T_{\widetilde{\lambda}}(Y)$ and there is an $H$-equivariant action of $A(\pi,\widetilde{I})\subset \mathrm{Aut}(\widehat{F}_2)$ on $\widetilde{R}$.  The kernel of the surjection $\overline{R} \to \widetilde{R}$ has stabilizer $G_{F'}$  in $G_F$ for some extension $F'$ of $F$ in $\overline{\mathbb{Q}}$.  
\begin{enumerate}
\item[i.] 
There is a finite index normal subgroup $A_{Y,\widetilde{I}}$ of $A(\pi,\widetilde{I})$ such that the action of $G_{F'} \cap \iota^{-1}(A_{Y,\widetilde{I}})$ on $\widetilde{R}$ agrees with the action of $A_{Y,\widetilde{I}}$ under  Bely\u{\i}'s embedding $G_{F'} \xrightarrow{\iota} \mathrm{Aut}(\widehat{F}_2)$.
\item[ii.] There is a finite extension $F''$ of $F'$ in $\overline{\mathbb{Q}}$ such that the action of $G_{F''} \cap \iota^{-1}(A(\pi, \widetilde{I}))$ agrees with the action of $A(\pi, \widetilde{I})$ under Bely\u{\i}'s embedding $G_{F''} \xrightarrow{\iota} \mathrm{Aut}(\widehat{F}_2)$.
\end{enumerate}
\end{corollary}

\begin{proof}  
The first statement follows  by letting $A_{Y,\widetilde{I}}$ be the intersection of $A(\pi,\widetilde{I})$ with the subgroup $A_Y$ appearing in the last sentence of the statement of Theorem \ref{thm:restriction}.  For the second, observe that since $A_{Y,\widetilde{I}}$ has finite index in $A(\pi,\widetilde{I})$, the kernel of the homomorphism $G_{F'} \cap \iota^{-1}(A(\pi,\widetilde{I}))  \to A(\pi,\widetilde{I})/A_{Y,\widetilde{I}}$ induced by $\iota$ has finite index in $G_{F'}\cap \iota^{-1}(A(\pi,\widetilde{I}))$. Therefore, this kernel equals $G_{F''} \cap \iota^{-1}(A(\pi, \widetilde{I}))$ for some finite extension $F''$ of $F'$.
\end{proof}

\begin{example}
\label{ex:nice}
Assume the notation from Corollary \ref{cor:Rtilde}. If $\overline{I}$ is a subgroup of $\overline{R}$ that is stabilized by the action of $H$ on $\overline{R}$, then $\overline{I}=\widetilde{I}/[R,R]$ for a normal subgroup $\widetilde{I}$ of $\widehat{F}_2$. Fix a prime number $p$. If $\overline{I}_p$ is the subgroup of $\overline{R}= \widehat{\mathbb{Z}}^{1 + |H|}$ generated by all factors isomorphic to $\mathbb{Z}_p$, then $\overline{I}_p=\widetilde{I}_p/[R,R]$ for a normal subgroup $\widetilde{I}_p$ of $\widehat{F}_2$. Moreover, the action of $G_F$ on $\overline{R}$ preserves $\overline{I}_p$.

Suppose now that we take $\widetilde{R}=R/\widetilde{I}$ to be the quotient of $\overline{R} = \widehat{\mathbb{Z}}^{1 + |H|}$ by the subgroup generated by all $\overline{I}_p$ for primes $p$ dividing $|H|$.  Then $\widetilde{I}$ is a normal subgroup of $\widehat{F}_2$, and the kernel of $\overline{R} \to \widetilde{R}$ is stable under all of $G_F$.  The action of $A(\pi)$ on $\overline{R}$ is $\widehat{\mathbb{Z}}$-linear, so $A(\pi)$ acts on $\widetilde{R}$ and $A(\pi,\widetilde{I}) = A(\pi)$.  If $\mathcal{L}'_H$ is the set of all primes $\ell$ not dividing $|H|$, then
$$\widetilde{R}=\prod_{\ell\in\mathcal{L}'_H}T_{\ell,\widetilde{\lambda}}(Y).$$
The extension class $\beta\in H^2(H,\widetilde{R})$ corresponding to $(\ref{eq:cut})$ is zero for this choice of $\widetilde{R}$. Therefore, it follows from Theorem  \ref{thm:main} that the image of the action of $A(\pi,\widetilde{I}) = A(\pi)$ on $\widetilde{R}$ is $\mathrm{Aut}_H(\widetilde{R})$.  Define
$$\widehat{\mathbb{Z}}' = \frac{\widehat{\mathbb{Z}}}{\prod_{p\not\in\mathcal{L}'_H}\mathbb{Z}_p}\cong \prod_{\ell\in\mathcal{L}'_H} \mathbb{Z}_\ell.$$
Then $\widetilde{R}$ is isomorphic to $\widehat{\mathbb{Z}}'[H]\oplus \widehat{\mathbb{Z}}'$ as an $H$-module since $|H|$ is a unit in $\widehat{\mathbb{Z}}'$ (see \cite[\S2]{Gruenberg}). Note that $\widehat{\mathbb{Z}}'[H] = \widehat{\mathbb{Z}}'\oplus M$ as $H$-module, where $M$ is the augmentation ideal of $\widehat{\mathbb{Z}}'[H]$. Hence, $\mathrm{Aut}_H(\widetilde{R})\cong \mathrm{GL}_2(\widehat{\mathbb{Z}}') \times \mathrm{Aut}_H(M)$. The right multiplication action of $\widehat{\mathbb{Z}}'[H]$ identifies the endomorphism ring $\mathrm{End}_{\widehat{\mathbb{Z}}'[H]}(M)$ with the quotient ring of the opposite ring $\widehat{\mathbb{Z}}'[H]^{\mathrm{op}}$ of $\widehat{\mathbb{Z}}'[H]$ modulo the ideal generated by the central idempotent $\frac{1}{|H|}\sum_{\sigma \in H} \sigma$.   Recall that the inversion $\sigma \mapsto \sigma^{-1}$ on $H$ extends to a ring  isomorphism of $\widehat{\mathbb{Z}}'[H]^{\mathrm{op}}$ with $\widehat{\mathbb{Z}}'[H]$.  In other words,
$$\mathrm{Aut}_H(\widetilde{R}) \cong \mathrm{GL}_2(\widehat{\mathbb{Z}}') \times 
\mathrm{GL}_1 \left( \frac{\widehat{\mathbb{Z}}'[H]}{\widehat{\mathbb{Z}}'[H]\cdot \left(\sum_{\sigma \in H} \sigma\right) } \right)$$
since $|H|$ is a unit in $\widehat{\mathbb{Z}}'$. Corollary \ref{cor:Rtilde} says that there is a finite index normal subgroup $A_{Y,\widetilde{I}}$ of $A(\pi,\widetilde{I}) = A(\pi)$ such that the action of   $G_F \cap \iota^{-1}(A_{Y,\widetilde{I}})$ on $\widetilde{R}$ extends to an action of $A_{Y,\widetilde{I}}$ under $\iota$.
\end{example}

\section{Actions on the Jacobian versus the generalized Jacobian}
\label{s:JacobianAlone}

Theorems \ref{thm:main1} and \ref{thm:restriction} show that the action of the absolute Galois group $G_{\mathbb{Q}}=\mathrm{Gal}(\overline{\mathbb{Q}}/\mathbb{Q})$ on the adelic Tate module of the generalized Jacobian can be virtually extended to an action of $\mathrm{Aut}(\widehat{F}_2)$.  In this section, we show that, in general, one indeed needs the full structure of the adelic Tate module of the generalized Jacobian and that the action of $G_{\mathbb{Q}}$ on the adelic Tate module of the Jacobian cannot be virtually extended to $\mathrm{Aut}(\widehat{F}_2)$. Moreover, we give a sufficient condition to distinguish between these two possibilities, which is also necessary when we have a Bely\u{\i} cover that is Galois. 

We use the notation from \S\ref{s:GT} and \S\ref{s:GalReps}.  In particular, we identify $\widehat{F}_2$ with the \'etale fundamental group $\pi_1(\mathbb{P}^1_{\overline{\mathbb{Q}}} - \{0,1,\infty\},\eta)$, and we view $G_{\mathbb{Q}}$ as a subgroup of $\mathrm{Aut}(\widehat{F}_2)$ via the injection $\iota$ of (\ref{eq:inject}).  As in \S\ref{s:GalReps}, let $\lambda:X \to \mathbb{P}^1_{\overline{\mathbb{Q}}}$ be a Bely\u{\i} cover, i.e. $X$ is a smooth projective curve over $\overline{\mathbb{Q}}$ and $\lambda$ is a non-constant morphism which is unramified outside $\{0,1,\infty\}$. Let $H$ be the Galois group of the Galois closure $\widetilde{\lambda}:Y  \to \mathbb{P}^1_{\overline{\mathbb{Q}}}$ of $\lambda$. Let $\pi:\widehat{F}_2 \to H$ be the homomorphism associated to $\widetilde{\lambda}$ and the choice of a geometric point of $Y$ over $\eta$.  

Let $\mathcal{J}\subset H$ be the Galois group of $Y$ over $X$, let $\mathcal{W}$ be the normalizer of $\mathcal{J}$ in $H$, and let $Z = Y/\mathcal{W}$.  We can factor $\lambda = \kappa\circ\delta$ where $\delta: X \to Z$ is a Galois cover with Galois group $\mathcal{D} = \mathcal{W}/\mathcal{J}$ and $\kappa:Z \to \mathbb{P}^1_{\overline{\mathbb{Q}}}$ is a Bely\u{\i} cover. For each point $j \in \kappa^{-1}(\{0,1,\infty\})$ let $j_X$ be a point over $j$ in $X$ and let $d_{j_X}$ be a generator for the inertia group of $j_X$ in $\mathcal{D}$.

Consider now the natural surjection $J_\lambda(X) \to J(X)$ from the generalized Jacobian of $X$ with respect to $\lambda^{-1}(\{0,1,\infty\})$ to   the Jacobian of $X$.  Recall that the points of $J_\lambda(X)$ over $\overline{\mathbb{Q}}$ are divisors of degree zero on $X$ prime to $\lambda^{-1}(\{0,1,\infty\})$ modulo divisors of the form $\mathrm{div}(f)$ for rational functions $f \in \overline{\mathbb{Q}}(X)^*$ such that $f(x)  = 1$ if  $x \in \lambda^{-1}(\{0,1,\infty\})$.  The principal divisor group of $X$ is isomorphic to $\overline{\mathbb{Q}}(X)^*/\overline{\mathbb{Q}}^*$.  Consider the homomorphism from the group of elements $g \in \overline{\mathbb{Q}}(X)^*$  that have no zeros or poles in $\lambda^{-1}(\{0,1,\infty\})$ to the product $\prod_{x \in \lambda^{-1}(\{0,1,\infty\})} \overline{\mathbb{Q}}^*$ which sends $g$ to the element of the product with component $g(x)$ at $x$.  This homomorphism induces an isomorphism between the algebraic group that is the kernel of $J_\lambda(X) \to J(X)$ and the split torus 
$$T :=\frac{\prod_{x \in \lambda^{-1}(\{0,1,\infty\})} \mathbb{G}_m}{\mathrm{diag}(\mathbb{G}_m)}.$$ 
The exact sequence
$$1 \to T \to J_\lambda(X) \to J(X) \to 1$$
gives an exact sequence of $\ell$-adic Tate modules
\begin{equation}
\label{eq:genjac2}
0 \to \frac{\prod_{x \in \lambda^{-1}(\{0,1,\infty\})} \mathbb{Z}_\ell(1)}{\mathrm{diag}(\mathbb{Z}_\ell(1))} \to T_{\ell,\lambda}(X) \to T_\ell(X) \to 0.
\end{equation}
This sequence is exact as a sequence of $G_F$-modules once $F$ is a number field over which $X$, $J(X)$, $J_\lambda(X)$ and the action of $\mathcal{D}$ on $X$ are defined.  

The set $\lambda^{-1}(\{0,1,\infty\}) = \delta^{-1} ( \kappa^{-1}(\{0,1,\infty\}) )$ is the $\mathcal{D}$-set formed by the disjoint union of the orbits of the $j_X$ as $j$ ranges over  $\kappa^{-1}(\{0,1,\infty\})$. We find that the tensor product of the left term in the sequence (\ref{eq:genjac2}) with $\mathbb{Q}_\ell$ over $\mathbb{Z}_\ell$ is isomorphic to 
\begin{equation}
\label{eq:leftterm2}
\frac{\bigoplus_{j \in \kappa^{-1}(\{0,1,\infty\})} \mathbb{Q}_\ell[\mathcal{D}/\langle d_{j_X} \rangle] }{\mathbb{Q}_\ell}
\end{equation}
as a $\mathbb{Q}_\ell[\mathcal{D}]$-module, where the denominator here is embedded into the numerator by sending $1 \in \mathbb{Q}_\ell$ to the sum of all the cosets of $\langle d_{j_X} \rangle$ as $j$ ranges over $\kappa^{-1}(\{0,1,\infty\}) $.  When we tensor the middle term of (\ref{eq:genjac2}) with $\mathbb{Q}_\ell$ over $\mathbb{Z}_\ell$, we get the $\mathbb{Q}_\ell[\mathcal{D}]$-module that is the group of $\mathcal{J}$-coinvariants of $\mathbb{Q}_\ell \oplus \mathbb{Q}_{\ell}[H]$ by Gasch\"utz's theorem (see \cite[\S2]{Gruenberg}).  

Let $A_X(\pi)$ be the finite index subgroup of $A(\pi)$ from the statement of Theorem \ref{thm:restriction}, so we have a well-defined action of $A_X(\pi)$ on $T_{\ell,\lambda}(X)$ as in Remark \ref{rem:ell}. Let $A_1$ be a finite index subgroup of $A_X(\pi)$.  We will say that the action of $A_1$ on $T_{\ell,\lambda}(X)$ descends to an action on the Tate module $T_\ell(X)$ if the action of $A_1$ respects the terms of the sequence $(\ref{eq:genjac2})$.  We now give a sufficient condition for this to occur which is also necessary when $X = Y$. In the latter case,   $A_X(\pi)=A(\pi)$ and $\lambda = \widetilde{\lambda}$.

\begin{theorem}
\label{thm:preserve2} Let $A_1$ be a finite index subgroup of $A_X(\pi)$. 
 A sufficient condition for the action of $A_1$ on $T_{\ell,\lambda}(X)$ to descend to an action on $T_\ell(X)$  is that for  every irreducible representation $V$ of $\mathcal{D}$ over $\overline{\mathbb{Q}}$, 
$$n_V :=  - \,\mathrm{dim}_{\overline{\mathbb{Q}}} (V^{\mathcal{D}})+ \sum_{j \in \kappa^{-1}(\{0,1,\infty\})} \mathrm{dim}_{\overline{\mathbb{Q}}}(V^{\langle d_{j_X}\rangle})  $$ 
is equal to either $0$ or 
$$m_V := \mathrm{dim}_{\overline{\mathbb{Q}}}(V^{\mathcal{D}})  + [H:\mathcal{W}]\cdot \mathrm{dim}_{\overline{\mathbb{Q}}}(V).$$   
This condition for $V$ holds if and only if  $\overline{\mathbb{Q}}_\ell \otimes_{\overline{\mathbb{Q}}} V$ occurs in exactly one of the $\overline{\mathbb{Q}}_\ell[\mathcal{D}]$-modules that result from tensoring either the left or right term of $(\ref{eq:genjac2})$ with $\overline{\mathbb{Q}}_\ell$ over $\mathbb{Z}_\ell$. If $X = Y$, then $\mathcal{D} = H=\mathcal{W}$ and the above sufficient condition is also necessary.  Furthermore, if $X = Y$, then the condition holds for all $V$ of dimension $1$.
\end{theorem}

\begin{proof}
It follows from Theorem \ref{thm:restriction} and Remark \ref{rem:ell} that $\mathbb{Q}_\ell \otimes_{\mathbb{Z}_\ell} T_{\ell,\lambda}(X)$ is isomorphic to the coinvariants $(\mathbb{Q}_\ell \otimes_{\mathbb{Z}_\ell} T_{\ell,\widetilde{\lambda}}(Y))_{\mathcal{J}}$.  By Gasch\"utz's theorem (see \cite[\S2]{Gruenberg}), $\mathbb{Q}_\ell \otimes_{\mathbb{Z}_\ell} T_{\ell,\widetilde{\lambda}}(Y)$ is isomorphic to $\mathbb{Q}_\ell \oplus \mathbb{Q}_\ell[H]$ as a $\mathbb{Q}_\ell[H]$-module.  Therefore, $\mathbb{Q}_\ell \otimes_{\mathbb{Z}_\ell} T_{\ell,\lambda}(X)$ is isomorphic to $\mathbb{Q}_\ell \oplus \mathbb{Q}_\ell[\mathcal{D}]^{[H:\mathcal{W}]}$ as a $\mathbb{Q}_\ell[\mathcal{D}]$-module.  

Let $V$ be an arbitrary irreducible representation of $\mathcal{D}$ over $\overline{\mathbb{Q}}$.  Then $m_V$ is the multiplicity of $V$ as a direct summand of $\overline{\mathbb{Q}}\oplus \overline{\mathbb{Q}}[\mathcal{D}]^{[H:\mathcal{W}]}$.  By Frobenius reciprocity and (\ref{eq:leftterm2}), $n_V$ is the multiplicity of $\overline{\mathbb{Q}}_\ell\otimes_{\overline{\mathbb{Q}}}V$ as a direct summand of the tensor product of the left term of (\ref{eq:genjac2}) with $\overline{\mathbb{Q}}_\ell$ over $\mathbb{Z}_\ell$.  Therefore, Theorem \ref{thm:preserve2} is equivalent to the statement that the action of $A_1$ respects the sequence $(\ref{eq:genjac2})$ if 
\begin{equation}
\label{eq:oy2}
n_V\in\{0,m_V\}\quad\mbox{for every irreducible  representation $V$ of $\mathcal{D}$ over $\overline{\mathbb{Q}}$},
\end{equation}
and if $X = Y$ then (\ref{eq:oy2}) is also necessary for $A_1$ to respect the sequence $(\ref{eq:genjac2})$.

Suppose first that the action of $A_1$ does not respect the sequence $(\ref{eq:genjac2})$. Since the terms of $(\ref{eq:genjac2})$ are free $\mathbb{Z}_\ell$-modules, the action of $A_1$ does not respect the terms of the sequence obtained by tensoring $(\ref{eq:genjac2})$ with $\overline{\mathbb{Q}}_\ell$ over $\mathbb{Z}_\ell$. Hence, we obtain a non-zero $\overline{\mathbb{Q}}_\ell[\mathcal{D}]$-module homomorphism from the tensor product with $\overline{\mathbb{Q}}_\ell$ over $\mathbb{Z}_\ell$ of the left term of (\ref{eq:genjac2}) to $\overline{\mathbb{Q}}_\ell\otimes_{\mathbb{Z}_\ell}T_\ell(X)$. Since the irreducible representations of $\mathcal{D}$ over $\overline{\mathbb{Q}}_\ell$ are all defined over $\overline{\mathbb{Q}}$, this means that there exists an irreducible representation $V$ of $\mathcal{D}$ over $\overline{\mathbb{Q}}$ such that $0 < n_V < m_V$, so (\ref{eq:oy2}) does not hold. 

Now suppose that $X = Y$, so that $\mathcal{J}$ is the trivial subgroup of $H$, and hence $\mathcal{W} = H$ and $\mathcal{D}=H$.  We suppose that the action of $A_1$ respects the terms of the sequence $(\ref{eq:genjac2})$, and we must show that (\ref{eq:oy2}) holds.  The action of $A_1$ preserves the terms of the sequence that results from tensoring $(\ref{eq:genjac2})$ with $\overline{\mathbb{Q}}_\ell$ over $\mathbb{Z}_\ell$. Since $T_{\ell,\lambda}(X)=T_{\ell,\widetilde{\lambda}}(Y)$ is the maximal pro-$\ell$ quotient of $\overline{R}=T_{\widetilde{\lambda}}(Y)$, we can write $T_{\ell,\lambda}(X)=R/\widetilde{I}_\ell$ for a characteristic subgroup $\widetilde{I}_\ell$ of $R$. By Theorem \ref{thm:main}, it follows that the image of the action of $A_X(\pi)=A(\pi,\widetilde{I}_\ell)$ on $T_{\ell,\lambda}(X)=R/\widetilde{I}_\ell$ consists of all $H$-equivariant automorphisms of  $T_{\ell,\lambda}(X)$ that preserve a certain extension class.  Since this extension class is annihilated by $\ell^a$ for a sufficiently large integer $a > 0$, it follows that if $f:T_{\ell,\lambda}(X) \to T_{\ell,\lambda}(X)$ is any $\mathbb{Z}_\ell[H]$-module endomorphism of $T_{\ell,\lambda}(X)$ and $\mathrm{id}_{T_{\ell,\lambda}(X)}$ is the identity automorphism of $T_{\ell,\lambda}(X)$, then $\mathrm{id}_{T_{\ell,\lambda}(X)} + \ell^a f$ is induced by the action of an element of $A_X(\pi)$. The kernel $A_2$ of the left multiplication action of $A_X(\pi)$ on the finitely many left cosets of $A_1$ in $A_X(\pi)$ is a finite index normal subgroup of $A_X(\pi)$ contained in $A_1$.  Let $n = [A_X(\pi):A_2]$.  Then for all sufficiently large $a > 0$, the binomial theorem shows that there is an endomorphism $f_1$ of $T_{\ell,\lambda}(X)$ with the following property. The map 
$$ \frac{1}{\ell^a} \left ( \left (\mathrm{id}_{T_{\ell,\lambda}(X)} + \ell^a f \right)^n - \mathrm{id}_{T_{\ell,\lambda}(X)} \right ) = n\,f + \ell^a f_1$$
is an endomorphism of $T_{\ell,\lambda}(X)$ lying in the $\mathbb{Q}_\ell$-subalgebra $B$ of $\mathrm{End}_{\mathbb{Q}_\ell[H]}(\mathbb{Q}_\ell \otimes_{\mathbb{Z}_\ell} T_{\ell,\lambda}(X))$ generated by automorphisms arising from the action of $A_2 \subset A_1$ on $T_{\ell,\lambda}(X)$.  Taking the limit as $a \to \infty$ and using the fact that $B$ is a closed $\mathbb{Q}_\ell$-subspace of a finite dimensional $\mathbb{Q}_\ell$-vector space, we see that $B$ is all of $\mathrm{End}_{\mathbb{Q}_\ell[H]}(\mathbb{Q}_\ell \otimes_{\mathbb{Z}_\ell} T_{\ell,\lambda}(X))$.  Therefore, every element of $\mathrm{End}_{\mathbb{Q}_\ell[H]}(\mathbb{Q}_\ell \otimes_{\mathbb{Z}_\ell} T_{\ell,\lambda}(X))$ preserves the tensor product of the left term of (\ref{eq:genjac2}) with $\mathbb{Q}_\ell$ over $\mathbb{Z}_\ell$.   It follows that  every element of $\mathrm{End}_{\overline{\mathbb{Q}}_\ell[H]}(\overline{\mathbb{Q}}_\ell \otimes_{\mathbb{Z}_\ell} T_{\ell,\lambda}(X))$ preserves the tensor product of the left term of (\ref{eq:genjac2}) with $\overline{\mathbb{Q}}_\ell$ over $\mathbb{Z}_\ell$. But this implies the statement (\ref{eq:oy2}) since the irreducible representations of $H$ over $\overline{\mathbb{Q}}_\ell$ are all defined over $\overline{\mathbb{Q}}$.
\end{proof}

\begin{remark}
\label{rem:IDunno} 
Let $\mathcal{D}_1$ be a subgroup of $\mathcal{D}$. Suppose $V$ is an irreducible representation of $\mathcal{D}$ over $\overline{\mathbb{Q}}$ for which the sufficient criterion described in Theorem \ref{thm:preserve2} does not hold.  Let $V_1$ be any irreducible representation of $\mathcal{D}_1$ over $\overline{\mathbb{Q}}$ that occurs in the restriction of $V$ to $\mathcal{D}_1$.  Then $\overline{\mathbb{Q}}_\ell\otimes_{\overline{\mathbb{Q}}}V_1$ occurs in both of the restrictions to $\mathcal{D}_1$  of the $\overline{\mathbb{Q}}_\ell[\mathcal{D}]$-modules that result from tensoring either the left or right term of (\ref{eq:genjac2}) with $\overline{\mathbb{Q}}_\ell$ over $\mathbb{Z}_\ell$.  This implies that if for some subgroup $\mathcal{D}_1$ of $\mathcal{D}$ the counterpart of the criterion in Theorem \ref{thm:preserve2} holds for all irreducible representations of $\mathcal{D}_1$ over $\overline{\mathbb{Q}}$, then this criterion holds for $\mathcal{D}$.  This forces the action of every finite index subgroup $A_1$ of $A_X(\pi)$ to descend to $T_\ell(X)$.
\end{remark}

We now give an example in which  $X=Y$ and the action of no finite index subgroup $A_1$ of $A_X(\pi)=A(\pi)$ descends to an action on $T_\ell(X)$.

\begin{example}
\label{ex:nopreserve}
Suppose $X = Y$ and that $\mathcal{D}=H$ is the alternating group $A_5$ of order $60$ on the letters $\{a_1,a_2,a_3,a_4,a_5\}$.  Recall that $d_{0_X}$ and $d_{1_X}$ are generators of inertia groups of points of $X=Y$ over $0$ and $1$. It follows that they generate $\mathcal{D}=H$ and that $d_{\infty_X} = (d_{0_X} d_{1_X})^{-1}$.  Suppose $d_{0_X}$ is the three-cycle $(a_1,a_2,a_3)$ in $\mathcal{D}$ and $d_{1_X}$ is the  five-cycle $(a_1,a_2, a_3,a_4,a_5)$, so that $d_{\infty_X} = (d_{0_X} d_{1_X})^{-1}$ is a five-cycle.    Define $A_4$ to be the alternating group of order $12$ on the letters $\{a_1,a_2,a_3,a_4\}$. The induction $\mathbb{Q}[A_5/A_4]$ of the trivial representation of $A_4$ to $A_5$ is then isomorphic to the direct sum of one copy of the trivial representation $\mathbb{Q}$ of $A_5$ with a four-dimensional absolutely irreducible representation $V$ of $A_5$.  It is easy to check that the invariants $V^{\langle d_{j_X}\rangle}$ are trivial for $j = 1, \infty$. On the other hand, the three-cycle $d_{0_X}$ has eigenvalues $1, 1, \zeta_3$ and $\zeta_3^2$ on $V$, so $\mathrm{dim}_{\mathbb{Q}}(V^{\langle d_{0_X}\rangle}) = 2$. It follows by Theorem \ref{thm:preserve2} that no finite index subgroup $A_1$ of $A_X(\pi) = A(\pi)$ respects the terms of the sequence $(\ref{eq:genjac2})$. Therefore,  while the Galois action of some finite index subgroup $G_F$ of $G_{\mathbb{Q}}$ on the $\ell$-adic Tate module $T_{\ell,\lambda}(X)$ can be extended to an action of $A(\pi)$,  the same cannot be said for the  $\ell$-adic Tate module $T_\ell(X)$.
\end{example}

We next give an example to show that the sufficient condition in Theorem \ref{thm:preserve2} is not in general necessary for the action of $A_X(\pi)$ on $T_{\ell,\lambda}(X)$ to descend to an action on $T_\ell(X)$.  

\begin{example}
\label{ex:badluck} 
Let $E$ be the elliptic curve with affine equation $y^3 = t (t-1)$, and let $\kappa: E \to \mathbb{P}^1_{\overline{\mathbb{Q}}}$ be the Bely\u{\i} cover defined by $(t,y) \mapsto t$ in affine coordinates.  Then $\kappa$ is a Galois cover and $\mathrm{Gal}(E/\mathbb{P}^1_{\overline{\mathbb{Q}}})=\mu_3$ is a cyclic group of order $3$, corresponding to the action of cube roots of unity on $E$ via the complex multiplication of $E$.  Suppose $p \ne 3 $ is a prime and that $\mathcal{J}$ is a cyclic subgroup of the $p$-torsion $E(\overline{\mathbb{Q}})[p]$  that is not stable under the action of $\mathrm{Gal}(E/\mathbb{P}^1_{\overline{\mathbb{Q}}})$.  Define $Y = E$ and let $\widetilde{\lambda}:Y \to \mathbb{P}^1_{\overline{\mathbb{Q}}}$ be the Bely\u{\i} Galois cover that is the composition of multiplication by $p$ with $\kappa$.  Define $X$ to be the quotient $Y/\mathcal{J}$.  Then there is an isogeny $\delta:X \to E$ such that the multiplication by $p$ map $Y \to E$ is the composition of the quotient morphism $Y \to X$ with $\delta$.  The composition of $\delta$ with $\kappa$ is a Bely\u{\i} cover $\lambda:X \to \mathbb{P}^1_{\overline{\mathbb{Q}}}$.  The Galois group $H = \mathrm{Gal}(Y/\mathbb{P}^1_{\overline{\mathbb{Q}}})$ is the semi-direct product $E(\overline{\mathbb{Q}})[p] \rtimes \mu_3$, and the normalizer of $\mathcal{J} = \mathrm{Gal}(Y/X) $ in $H$ is $\mathcal{W} = E(\overline{\mathbb{Q}})[p]$.  Thus $\mathcal{D} = \mathcal{W}/\mathcal{J} = \mathrm{Gal}(X/E)$ is the cyclic order $p$ group associated to the isogeny $\delta$. Each of the points $0, 1, \infty \in \mathbb{P}^1_{\overline{\mathbb{Q}}}$ are totally ramified with respect to $\kappa:E \to \mathbb{P}^1_{\overline{\mathbb{Q}}}$, and the points over them are split under $\delta:X \to E$.  Thus if $V$ is the trivial one-dimensional representation of $\mathcal{D}$, then
$$n_V =  -\, \mathrm{dim}_{\overline{\mathbb{Q}}} (V^{\mathcal{D}})+ \sum_{j \in \kappa^{-1}(\{0,1,\infty\})} \mathrm{dim}_{\overline{\mathbb{Q}}}(V^{\langle d_{j_X}\rangle}) = -1 + 3 = 2  $$
and
$$m_V = \mathrm{dim}_{\overline{\mathbb{Q}}}(V^{\mathcal{D}})  + [H:\mathcal{W}]\cdot \mathrm{dim}_{\overline{\mathbb{Q}}}(V) = 1 + 3 = 4.$$
Hence it is not true that $n_V$ is either $0$ or $m_V$, and the condition in Theorem 5.1 does not hold.  Let us now check that nevertheless the action of $A_X(\pi)$ descends to $T_\ell(X)$, i.e. that this action preserves the terms of (\ref{eq:genjac2}).  For this we use that the $p$-isogeny $\delta:X \to E$ is induced by the trace element of the group ring $\mathbb{Z}_\ell[\mathcal{D}]$ and induces a commutative diagram
\begin{equation}
\label{eq:bigarray}
\xymatrix{
0 \ar[r] & \displaystyle\frac{\prod_{x \in \lambda^{-1}(\{0,1,\infty\})} \mathbb{Z}_\ell(1)}{\mathrm{diag}(\mathbb{Z}_\ell(1))} \ar[d] \ar[r] & T_{\ell,\lambda}(X) \ar[d] \ar[r] & T_\ell(X) \ar[d] \ar[r] & 0\;\\
0 \ar[r] & \displaystyle\frac{\prod_{x \in \kappa^{-1}(\{0,1,\infty\})} \mathbb{Z}_\ell(1)}{\mathrm{diag}(\mathbb{Z}_\ell(1))} \ar[r] & T_{\ell,\kappa}(E) \ar[r] &T_\ell(E) \ar[r] & 0.
}
\end{equation}
Since the actions of $A_X(\pi)$ and $\mathcal{D}$ on $T_{\ell,\lambda}(X)$ commute, the middle vertical arrow in (\ref{eq:bigarray}) is equivariant for the action of $A_X(\pi)$.   Suppose that the terms of the top row of (\ref{eq:bigarray}) are not stable under the action of $A_X(\pi)$.  Then there exists $\alpha \in A_X(\pi)$ and there exists $c \in T_{\ell,\lambda}(X)$ with non-zero image in $T_\ell(X)$ such that $\alpha \cdot c$ has trivial image in $T_\ell(X)$. The right vertical arrow in (\ref{eq:bigarray}) is injective, since the elements of $\mathcal{D}$ act by translations on the elliptic curve $X$ and thus they act trivially on the torsion free module $T_{\ell}(X)$.  Therefore, the image $c'$ of $c$ under the middle vertical arrow in (\ref{eq:bigarray}) is an  element of $T_{\ell,\kappa}(E) $ with non-zero image in $T_\ell(E)$ such that $\alpha\cdot c'$ has trivial image in $T_\ell(E)$. This means that the terms of the bottom row of (\ref{eq:bigarray}) are not respected by the action of $A_X(\pi)$.  However, we can now apply Theorem \ref{thm:preserve2} to the cyclic morphism $\kappa:E \to \mathbb{P}^1_{\overline{\mathbb{Q}}}$ with Galois group $\mu_3$.  Since $\kappa$ is Galois and all the irreducible characters of $\mu_3$ over $\overline{\mathbb{Q}}$ have dimension $1$, the condition in Theorem \ref{thm:preserve2} is automatically satisfied. Hence the action of $A_X(\pi)$ respects the terms in the bottom row of (\ref{eq:bigarray}), which is a contradiction. In other words, the action of $A_X(\pi)$ on $T_{\ell,\lambda}(X)$ descends to an action on $T_\ell(X)$.
\end{example}
 
We end this section by proving Theorem \ref{thm:irred} of the introduction.  

\begin{proof}[Proof of Theorem \ref{thm:irred}.]
As in the statement of the theorem, suppose $t> 1$ is an odd integer. Let $\zeta$ be a primitive $t^{\mathrm{th}}$ root of unity, and let $E=E_\zeta$ be the elliptic curve with affine equation
$$y^2 = x (x-1) (x-\zeta).$$
Let $\lambda:E \to \mathbb{P}^1_{\overline{\mathbb{Q}}}$ be the Bely\u{\i} cover defined by $(x,y)\mapsto x^t$ in affine coordinates. Then $\lambda$ factors as $\lambda = \kappa\circ\delta$ where $\delta:E \to \mathbb{P}^1_{\overline{\mathbb{Q}}}$ is given by $(x,y)\mapsto x$, and $\kappa: \mathbb{P}^1_{\overline{\mathbb{Q}}}\to \mathbb{P}^1_{\overline{\mathbb{Q}}}$ is given by $x\mapsto x^t$, both in affine coordinates. 
Over $\overline{\mathbb{Q}}$, $\kappa$ is a cyclic cover of order $t$.  Since $\delta:E\to \mathbb{P}^1_{\overline{\mathbb{Q}}}$ is quadratic, the Galois group $H$ of the Galois closure $\widetilde{\lambda}:Y \to \mathbb{P}^1_{\overline{\mathbb{Q}}}$ of $\lambda$  is a semi-direct product of a cyclic group $\mathbb{Z}/t$ of order $t$ with a normal elementary abelian 2-group $G$.  The Galois group $\mathcal{J}$ of the natural morphism $z:Y \to E$ is an index two subgroup of $G$.  We can identify $G$ with the Galois group of the morphism $\delta \circ z:Y \to \mathbb{P}^1_{\overline{\mathbb{Q}}}$, where $\widetilde{\lambda} = \kappa \circ \delta \circ z$.  Note that $G$ is contained in the normalizer $\mathcal{W}$ of $\mathcal{J}$ in $H$. As before, we let $\pi: \widehat{F}_2 \to H$ be the surjective homomorphism associated to $\widetilde{\lambda}$. 

We now apply Theorem \ref{thm:restriction} and Remark \ref{rem:ell} to $X=E$.  We obtain a finite index subgroup $A_E(\pi)$ of $A(\pi)$ so that $A_E(\pi)$ acts on the $\ell$-adic Tate module $T_{\ell,\lambda}(E)$. Let $F_1\supset \mathbb{Q}(\zeta)$ be a number field over which $E=E_\zeta$, $Y$, $J_{\lambda}(E)$, $J_{\widetilde{\lambda}}(Y)$ and the action of $H$ on $Y$ are defined. Then there exists a finite index normal subgroup $A_E$ of $A_E(\pi)$ such that the action of $G_{F_1} \cap \iota^{-1}(A_E)$ on $T_{\ell,\lambda}(E)$ agrees with the action of $A_E$ under  Bely\u{\i}'s embedding $G_{F_1} \xrightarrow{\iota} \mathrm{Aut}(\widehat{F}_2)$. Since $G_{F_1} \cap \iota^{-1}(A_E)$ is a finite index subgroup of $G_{F_1}$, there exists a number field $F$ containing $F_1$ such that $G_F=G_{F_1} \cap \iota^{-1}(A_E)$. In other words, $A_E$ contains the image of $G_F$ under $\iota$.

We next use the criterion of Theorem \ref{thm:preserve2} and Remark \ref{rem:IDunno} to show that the action of $A_E(\pi)$, and hence the action of $A_E$, on $T_{\ell,\lambda}(E)$ descends to an action on $T_\ell(E)$.  Let $\mathcal{D}_1$ be the order two subgroup $G/\mathcal{J}$ of $\mathcal{D} = \mathcal{W}/\mathcal{J}$.  Let $V_1$ be an irreducible representation of $\mathcal{D}_1$ over $\overline{\mathbb{Q}}$. Then the character of $V_1 $ is either the trivial character or the order two character of $\mathcal{D}_1$.  By Remark \ref{rem:IDunno}, it suffices to show that $\overline{\mathbb{Q}}_\ell \otimes_{\overline{\mathbb{Q}}} V_1$ occurs in at most one of the restrictions to $\mathcal{D}_1$ of the $\overline{\mathbb{Q}}_\ell[\mathcal{D}]$-modules that result from tensoring either the left or right term of (\ref{eq:genjac2}) with $\overline{\mathbb{Q}}_\ell$ over $\mathbb{Z}_\ell$.   Here the action of $\mathcal{D}_1$ on these terms results from the elliptic involution associated to $\delta:E \to \mathbb{P}^1_{\overline{\mathbb{Q}}}$.  This involution acts trivially on each of the points in $\lambda^{-1}(\{0,1,\infty\}) \subset E$, so it acts trivially on the left term of (\ref{eq:genjac2}).  On the other hand, the elliptic involution $(x,y) \mapsto (x,-y)$ on $E$ acts as multiplication by $-1$ on $T_\ell(E)$.  Therefore,  each $V_1$ can occur in only one of the above restrictions, which implies that $A_E(\pi)$ acts on $T_\ell(E)$.

Finally, suppose that the value $\phi(t)$ of Euler's phi function on $t$ is larger than $24$. Let $\tau_\ell: A_E \to \mathrm{GL}(\mathbb{Q}_\ell \otimes_{\mathbb{Z}_\ell} T_\ell(E))$ be the representation of $A_E$ over $\mathbb{Q}_\ell$ that is associated to the action of $A_E$ on $T_\ell(E)$. It remains to check that the restriction of $\tau_\ell$ to any subgroup between $A_E$ and $\iota(G_F)$ is absolutely irreducible and non-abelian. For this it will suffice to show these properties for the representation of $G_F$ over $\mathbb{Q}_\ell$ that is associated to the action of $G_F$ on $T_\ell(E)$. By a result of Serre  (see \cite[Chapter IV]{SerreElliptic1968}), this will be true as long as $E$ does not have complex multiplication. Using the formula for the $j$-invariant of $E=E_\zeta$ in \cite[Prop. III.1.7]{Silverman}, we see that 
$$j(E)=256\,\frac{(\zeta^2-\zeta+1)^3}{\zeta^2 (\zeta-1)^2}.$$
In particular, $j(E)$ lies in $\mathbb{Q}(\zeta)\subset F$. If $E$ has complex multiplication then $j(E)$ lies in the Hilbert class field of an imaginary quadratic field, which is a dihedral extension of $\mathbb{Q}$. However, the maximal abelian quotient of such a dihedral group has order 4.  Therefore, we only need to ensure that $j(E)$ generates an extension of $\mathbb{Q}$ of degree larger than 4. However, this is the case provided the minimal polynomial of $\zeta$ has degree greater than 24, which happens if and only if $\phi(t)>24$. This completes the proof of Theorem \ref{thm:irred}.
\end{proof}

\section{The case of elliptic curves that are Galois covers of $\mathbb{P}^1_{\overline{\mathbb{Q}}} - \{0,1,\infty\}$.}
\label{s:examples}

Suppose $X$ is a Galois $H$-cover of $\mathbb{P}^1_{\overline{\mathbb{Q}}}$ that is unramified outside $\{0, 1, \infty\}$ and associated to a surjection $\pi:\widehat{F}_2\to H$. Theorem \ref{thm:preserve2} shows that the action of a finite index subgroup of $G_{\mathbb{Q}}$ on the adelic Tate module of the Jacobian of $X$ cannot, in general, be extended to an action of the finite index subgroup $A(\pi)$ of $\mathrm{Aut}(\widehat{F}_2)$.  In this section, we will show that such an extension does always exist when $X$ has genus $1$, and that all such $X$ must be CM elliptic curves.  We will also show that when $\iota: G_{\mathbb{Q}} \to \mathrm{Aut}(\widehat{F}_2)$ is the Bely\u{\i} embedding, the action of $\iota(G_{\mathbb{Q}}) \cap A(\pi)$ on the adelic Tate module of the generalized Jacobian of $X$ has infinite index in the group of automorphisms generated by $A(\pi)$, while the corresponding index is finite if we replace the generalized Jacobian by the Jacobian of $X$.

Throughout this section we will make the following hypothesis:

\begin{hypo}
\label{hyp:nice}  
The morphism $\lambda: X\to \mathbb{P}^1_{\overline{\mathbb{Q}}}$ is an $H$-cover unramified outside $ \{0,1,\infty\}$, associated to a surjection $\pi:\widehat{F}_2 \to H$ and an embedding of $\overline{\mathbb{Q}}(X)$ into an algebraic closure of $\overline{\mathbb{Q}}(t) = \overline{\mathbb{Q}}(\mathbb{P}^1_{\overline{\mathbb{Q}}})$.  The genus of $X$ is $1$.  There are points $0_X$, $1_X$ and $\infty_X$ of $X$ over $0$, $1$ and $\infty$ on $\mathbb{P}^1_{\overline{\mathbb{Q}}}$ such that the corresponding inertia groups $I_{0_X}$, $I_{1_X}$ and $I_{\infty_X}$ satisfy $| I_{0_X} | \ge | I_{1_X} | \ge | I_{\infty_X}|$.
\end{hypo}

Note that we can always arrange for the last property  to hold by composing $\lambda$ with an automorphism of $\mathbb{P}^1_{\overline{\mathbb{Q}}}$ which permutes $\{0,1,\infty\}$. In the notation of the previous two sections, we are assuming that $X=Y$ is its own Galois closure over $\mathbb{P}^1_{\overline{\mathbb{Q}}}$ and that $\lambda=\widetilde{\lambda}$, which implies that $A_X(\pi)=A(\pi)$.

Here are some examples which are easily checked using Hurwitz's theorem (see \cite[\S IV.2]{Hartshorne1977}).

\begin{example}
\label{ex:complex}  
Let $d \in \{3, 6, 4\}$.  Suppose $(a,b,d)$ is a triple of integers in the set $$\{(1,1,3), (2,2,3), (1,2,6), (5,4,6), (1,1,4),(3,3,4)\}.$$
Let $X$ be the curve with affine equation $y^d = t^a(t-1)^b$ and let $\lambda_{a,b,d}:X \to \mathbb{P}^1_{\overline{\mathbb{Q}}}$ be the map sending $(y,t)$ to $t$.  Fix a root of unity $\zeta_d$ of order $d$ in $\overline{\mathbb{Q}}$.   Then Hypothesis (\ref{hyp:nice}) holds with $\lambda = \lambda_{a,b,d}$ and the generator $1$ of $H_{a,b,d} = \mathbb{Z}/d$ sending $y$ to $\zeta_d\, y$.  Note that the isomorphism class of $X$ as an $H_{a,b,d}$-cover of $\mathbb{P}^1_{\overline{\mathbb{Q}}}$ depends on the  choice of $\zeta_d$.  There is a unique point $0_X$ over $0$ on $X$, and $I_{0_X} = H_{a,b,d}$. The ordered pair $(| I_{1_X}|, | I_{\infty_X}|)$  equals $(3,3)$ if $d = 3$, $(3,2)$ if $d = 6$ and $(4,2)$ if $d = 4$. The elliptic curve $X$ with origin $0_X$ is isomorphic to an elliptic curve over $\overline{\mathbb{Q}}$ with complex multiplication by the ring of integers $\mathbb{Z}[\zeta_d]$ of $\mathbb{Q}(\zeta_d)$, and all such elliptic curves are isomorphic.
\end{example}

\begin{proposition}
\label{prop:ellconstruct}  
Let $X$ be as in Hypothesis \ref{hyp:nice}, so that $X$ is an elliptic curve with origin $0_X$.   Then $\lambda$ can be factored as a finite \'etale isogeny $\lambda_1:X \to X$ followed by one of the morphisms $\lambda_{a,b,d}:X \to \mathbb{P}^1_{\overline{\mathbb{Q}}}$ defined in Example \ref{ex:complex}.  The group $H$ is the semi-direct product of $H_{a,b,d} = \mathbb{Z}/d$ with the kernel $\mathcal{J}$ of $\lambda_1$.  We can identify $\mathcal{J}$ with an $H_{a,b,d}$-stable submodule of the torsion points of $X$ as an elliptic curve.  Conversely, any such finite $H_{a,b,d}$-stable subgroup of torsion points can be taken to be the kernel of a $\lambda_1$ of the above kind, leading to a cover $\lambda:X \to \mathbb{P}^1_{\overline{\mathbb{Q}}}$ as in Hypothesis \ref{hyp:nice}.  The action of $H_{a,b,d}$ on torsion points agrees with the action of $\langle\zeta_d\rangle$ when we identify the Kummer action of the generator $1$ of $H_{a,b,d}$ with complex multiplication by $\zeta_d$.  
\end{proposition}

\begin{proof} 
By Hurwitz's Theorem,
$$0 = 2g(X) - 2 = |H| \cdot (-2) + \sum_{j = 0, 1 , \infty} [H:I_{j_X}] \cdot ( | I_{j_X} |- 1).$$
It follows that 
\begin{equation}
\label{eq:whatup}
1 = \sum_{j = 0, 1, \infty} 1/| I_{j_X} |.
\end{equation}
This forces $| I_{j_X} | \ge 3$ for some $j$, so $d = |I_{0_X}|\ge 3$ since $I_{0_X}$ has the largest order of any inertia group.   Fix an identification of $I_{0_X}$ with $\mathbb{Z}/d$.  The generator $1$ of $I_{0_X}$ then acts on $X$ via complex multiplication by a root of unity $\zeta_d$ of order $d$.  This forces $d \in \{3,6,4\}$.   Now (\ref{eq:whatup}) together with the inequalities $d = |I_{0_X}| \ge |I_{1_X}| \ge |I_{\infty_X}|$ force $(|I_{1_X}|, |I_{\infty_X}|)$ to be $(3,3)$ if $d = 3$, $(3,2)$ if $d = 6$, and $(4,2)$ if $d = 4$.  If $ d= 6$ or $d = 4$, the group $I'_{0_X}$ of automorphisms of $X$ fixing $0_X$ is equal to $I_{0_X}$.  If $d = 3$, $I'_{0_X}$ is cyclic of order $6$ and generated by $I_{0_X}$ together with the multiplication by  $-1$ map $z:X \to X$.

Let $\mathcal{T}$ be the group of elements of $H$ which are translations $X \to X$ relative to the group law of $X$. Suppose $h$ is an arbitrary element of $H$.  Then $h(0_X) = t(0_X)$ for some translation $t:X \to X$, where we do not claim at this point that $t \in \mathcal{T}$.  Moreover, $t^{-1} \circ h = \iota:X \to X$ is an automorphism of $X$ which fixes $0_X$.  So $\iota$ is a unit in the endomorphism ring of $X$ as an elliptic curve.  If $d = 6$ or $4$, all such units are produced by powers of a generator of $I_{0_X}$, so it follows that $\iota \in I_{0_X}$ and $t \in \mathcal{T}$.  Thus in this case, $H$ is the semi-direct product of $I_{0_X}$ with the normal subgroup $\mathcal{T}$ of $H$. 

Suppose now that $d = 3$.  Then $\iota \in I_{0_X}$ or $z^{-1} \circ \iota \in I_{0_X}$.  We find that in this case, $h = t \circ \iota = \widetilde{t} \circ \widetilde{\iota}$ where either $t = \widetilde{t} \in \mathcal{T}$ and $\iota = \widetilde{\iota} \in I_{0_X}$, or $\widetilde{t} = t\circ z$ and $\widetilde{\iota} = z^{-1} \circ \iota \in I_{0_X}$.  Suppose the alternative $\widetilde{t}  = t \circ z$ occurs.  Then  $\widetilde{t} = t \circ z = h \circ \widetilde{\iota}^{-1}$ lies in $H$.  In this case there would be a point $\tau \in X$ such that  $\widetilde{t}(P) = -P + \tau$ for all $P \in X$.  Hence, $\widetilde{t}(\tau') = \tau'$ for all points $\tau' \in X$ satisfying  $ \tau' + \tau'= \tau$  with respect to the group law on $X$. Since $\widetilde{t}^2$ is the identity, we would then have some points of $X$ with inertia groups of even order, which we have shown does not occur when $d = 3$.  So in fact, $h = t \circ \iota$ with $t \in \mathcal{T}$ and $\iota \in I_{0_X}$. This implies $H$ is the semi-direct product of $I_{0_X}$ with $\mathcal{T}$ in all cases.

Recall that the action of the generator $1$ of $I_{0_X} = \mathbb{Z}/d$ on $X$ defines the complex multiplication of $X$ corresponding to $\zeta_d$. Since $\mathcal{T}$ is stable under the conjugation action of $I_{0_X}$, $\mathcal{T}$ is the group of translations associated to  a subgroup $\mathcal{J}(\mathcal{T})$ of torsion points on $X$ which is stable under the action of $\mathbb{Z}[\zeta_d]$.  We have a $\mathcal{T}$-Galois isogeny  $\lambda_1:X \to X' = X/\mathcal{T} = X/\mathcal{J}(\mathcal{T})$, and $X'$ is an elliptic curve over $\overline{\mathbb{Q}}$ with complex multiplication by $\mathbb{Z}[\zeta_d]$.  Therefore $X'$ is isomorphic to $X$.  The morphism $\lambda_2: X' = X/\mathcal{T} \to \mathbb{P}^1_{\overline{\mathbb{Q}}}$ induced by $\lambda:X \to \mathbb{P}^1_{\overline{\mathbb{Q}}}$ defines an $H/\mathcal{T}$ Galois cover of $\mathbb{P}^1_{\overline{\mathbb{Q}}}$ which is unramified outside $\{0,1,\infty\}$. Furthermore the order of the inertia group of $H/\mathcal{T}$ at a point of $X'$ over $j \in \{0,1,\infty\}$ must be the same as the order of $I_{j_X}$ since $\lambda_1:X \to X'$ is \'etale.  Since $X'$ is isomorphic to $X$ and $H/\mathcal{T}$ is isomorphic to $I_{0_X}$, Kummer theory shows that $\lambda_2:X' \to \mathbb{P}^1_{\overline{\mathbb{Q}}}$ must be $H/\mathcal{T} $-isomorphic to a unique $\mathbb{Z}/d$-cover  $\lambda_{a,b,d}:X \to \mathbb{P}^1_{\overline{\mathbb{Q}}}$ appearing in Example \ref{ex:complex} when we identify $H/\mathcal{T}$ with $I_{0_X}$ and $I_{0_X}$ with $\mathbb{Z}/d$.  The $H$-cover $X \to \mathbb{P}^1_{\overline{\mathbb{Q}}}$ we started with then results from the isogeny $X \to X' = X/\mathcal{T}$ followed by $\lambda_{a,b,d}$.
\end{proof}

\begin{theorem}
\label{thm:ellthm} 
With the notation of Proposition \ref{prop:ellconstruct}, the action of $A(\pi)$ on the Tate module $T_{\lambda}(X)$ of the generalized Jacobian descends to an action of $A(\pi)$ on $T(X)$ which respects the action of $\mathbb{Z}[\zeta_d]$.  Let $F$ be a number field such that $\iota(G_{\mathbb{Q}}) \cap A(\pi) = \iota(G_F)$. The action of $\iota(G_{\mathbb{Q}}) \cap A(\pi)$ on $T_{\lambda}(X)$ defines an infinite index subgroup of the group of automorphisms of $T_{\lambda}(X)$ induced by $A(\pi)$. The action of $\iota(G_{\mathbb{Q}}) \cap A(\pi)$ on $T(X)$ defines a finite index subgroup of the group of automorphisms of  $T(X)$ induced by $A(\pi)$.  
\end{theorem}

\begin{proof} 
By Proposition \ref{prop:ellconstruct}, $X$ is an elliptic curve with complex multiplication by $\mathbb{Z}[\zeta_d]$ for an integer $d \in \{3,4,6\}$.  The group $H$ is the semi-direct product of a group $\mathcal{T}$ of translations on $X$ with a cyclic group $I$ of order $d$, with a generator of $I$ acting by complex multiplication by $\zeta_d$.  The action of $\mathcal{T}$ on the adelic Tate module $T(X)$ is trivial.  So the action of $H$ on $T(X)$ factors through $H/\mathcal{T} = I$ and corresponds to the action of complex multiplication.    Therefore, in the sequence (\ref{eq:genjac2}) the action of $H$ on $\overline{\mathbb{Q}}_\ell \otimes_{\mathbb{Z}_\ell} T_{\ell}(X)$ gives two (non-trivial) one dimensional characters.  Hence, by Theorem \ref{thm:preserve2},  the action of $A(\pi)$ on $ T_{\lambda}(X) $ descends to an action on $T(X)$. 
 
Therefore, we have an action of $A(\pi)$ on the submodule 
$$\frac{\prod_{x \in \lambda^{-1}(\{0,1,\infty\})} \mathbb{Z}_\ell(1)}{\mathrm{diag}(\mathbb{Z}_\ell(1))}$$
of $T_{\ell,\lambda}(X)$ appearing in the sequence (\ref{eq:genjac2}).  This submodule contains the rank one $\mathbb{Z}_\ell$-modules $M_j$ with trivial action by $H$ that are the images of $\mathbb{Z}_\ell(1)$ diagonally embedded in
$$\prod_{x \in \lambda^{-1}(j)} \mathbb{Z}_\ell(1)$$
for $j \in \{0, 1, \infty\}$. In particular, $\mathbb{Q}_\ell \otimes_{\mathbb{Z}_\ell} M_0$ and $\mathbb{Q}_\ell \otimes_{\mathbb{Z}_\ell} M_1$ define two distinct one-dimensional $\mathbb{Q}_\ell[H]$-submodules of $\mathbb{Q}_\ell \otimes_{\mathbb{Z}_\ell} T_{\ell,\lambda}(X)$ with trivial $H$-action.  Since $\mathbb{Q}_\ell \otimes_{\mathbb{Z}_\ell} T_{\ell,\lambda}(X) \cong \mathbb{Q}_\ell \oplus \mathbb{Q}_\ell[H]$, there exists a nilpotent $\mathbb{Q}_\ell[H]$-module endomorphism $f$ of $\mathbb{Q}_\ell \otimes_{\mathbb{Z}_\ell} T_{\ell,\lambda}(X)$ that sends $\mathbb{Q}_\ell \otimes_{\mathbb{Z}_\ell} M_0$ isomorphically to  $\mathbb{Q}_\ell \otimes_{\mathbb{Z}_\ell} M_1$ and that sends all other irreducible $\mathbb{Q}_\ell[H]$-module summands of $\mathbb{Q}_\ell \otimes_{\mathbb{Z}_\ell} T_{\ell,\lambda}(X)$ to zero. Multiplying $f$ with a sufficiently large positive integer, say $m_1$, we obtain that $m_1\cdot f$ is a nilpotent $\mathbb{Z}_\ell[H]$-module endomorphism of $T_{\ell,\lambda}(X)$ that sends $M_0$ to a non-zero submodule of $M_1$. If $m_2$ is the largest power of $\ell$ dividing $|H|$, it then follows that $m_1m_2\cdot f$ is a nilpotent $\mathbb{Z}_\ell[H]$-module endomorphism of $T_{\ell,\lambda}(X)$ that sends $M_0$ to a non-zero submodule of $M_1$ and that sends any extension class in $H^2(H,T_{\ell,\lambda}(X))$ to zero. Since $T_{\ell,\lambda}(X)$ is the maximal pro-$\ell$ quotient of $\overline{R}=T_{\lambda}(X)$, we can write $T_{\ell,\lambda}(X)=R/\widetilde{I}_\ell$ for a characteristic subgroup $\widetilde{I}_\ell$ of $R$. It follows from Theorem \ref{thm:main} that $\mathrm{id}_{T_{\ell,\lambda}(X)}+m_1m_2\cdot f$ is an automorphism of $T_{\ell,\lambda}(X)$ that is induced by an element of $A(\pi)=A(\pi,\widetilde{I}_\ell)$. Moreover, $\mathrm{id}_{T_{\ell,\lambda}(X)}+m_1m_2\cdot f$ generates a subgroup of this automorphism group that is isomorphic to $\mathbb{Z}$ and every non-trivial element of this subgroup sends $M_0$ to a non-zero submodule of $M_1$.

Since $A(\pi)$ is a finite index subgroup of $\mathrm{Aut}(\widehat{F}_2)$, it follows that there exists a number field $F$ such that $\iota(G_{\mathbb{Q}}) \cap A(\pi) = \iota(G_F)$. Since $G_F$ cannot send points over $0$ to points over $1$ or $\infty$, the action of $G_F=G_{\mathbb{Q}}\cap \iota^{-1}(A(\pi))$ on $T_{\ell,\lambda}(X)$ preserves $M_0$ and $M_1$.  By the above construction of $\mathrm{id}_{T_{\ell,\lambda}(X)}+m_1m_2\cdot f$, this implies that the action of $\iota(G_{\mathbb{Q}}) \cap A(\pi)$ on $T_{\lambda}(X)$ defines an infinite index subgroup of the group of automorphisms of $T_{\lambda}(X)$ induced by $A(\pi)$. 

We now compare the actions of $G_F$ and $A(\pi)$ on $T(X)$.   The action of $\mathbb{Z}[\zeta_d]$ on $X$ by complex multiplication makes $T(X)$ a rank one free module for $\widehat{\mathbb{Z}}\otimes_{\mathbb{Z}} \mathbb{Z}[\zeta_d] = \widehat{\mathbb{Z}}[\zeta_d]$.  The action of $A(\pi)$ on $T(X)$ respects the action of $H$, and the action of $H$ corresponds to the action of complex multiplication.  On picking a basis for $T(X)$ as a rank one free module for $\widehat{\mathbb{Z}}[\zeta_d]$, we see that the action of $A(\pi)$ on $T(X)$ is defined by a homomorphism $\chi_A: A(\pi) \to \widehat{\mathbb{Z}}[\zeta_d]^*$.  Similarly, the action of $G_F$ on  $T(X)$ is defined by a homomorphism $\chi_F:G_F \to \widehat{\mathbb{Z}}[\zeta_d]^*$.  When $\iota:G_{\mathbb{Q}} \to \mathrm{Aut}(\widehat{F}_2)$ is the Bely\u{\i}  embedding, we know by Corollary \ref{cor:Rtilde} that, after enlarging $F$ by a finite extension, we have $\chi_F = \chi_A \circ \iota|_{G_F}$.  Hence, the image of $\chi_F$ is contained in the image of $\chi_A$.  On the other hand, the main theorem of complex multiplication (see \cite[Thm. 5.4]{ShimuraBook}) shows that the image of $\chi_F$ has finite index in $\widehat{\mathbb{Z}}[\zeta_d]^*$.   Therefore, the image of $\chi_F$ has finite index in the image of $\chi_A$ and the proof is complete.
 \end{proof}

\bibliographystyle{plain}
\bibliography{GTRepsShort} 

\begin{thebibliography}{10}

\bibitem{Belyi1}
G.~V. Bely\u{\i}.
\newblock Galois extensions of a maximal cyclotomic field.
\newblock {\em Izv. Akad. Nauk SSSR Ser. Mat.}, 43(2):267--276, 1979; English
  translation in \textit{Math. USSR Izv.}, 14:247--256, 1980.

\bibitem{Belyi}
G.~V. Bely\u{\i}.
\newblock On extensions of the maximal cyclotomic field having a given
  classical {G}alois group.
\newblock {\em J. Reine Angew. Math.}, 341:147--156, 1983.

\bibitem{D}
V.~G. Drinfel'd.
\newblock On quasitriangular quasi-{H}opf algebras and on a group that is
  closely connected with {${\rm Gal}(\overline{\bf Q}/{\bf Q})$}.
\newblock {\em Algebra i Analiz}, 2(4):149--181, 1990.

\bibitem{FJ}
M.~D. Fried and M.~Jarden.
\newblock {\em Field arithmetic}, volume~11 of {\em Ergebnisse der Mathematik
  und ihrer Grenzgebiete. 3. Folge. A Series of Modern Surveys in Mathematics
  [Results in Mathematics and Related Areas. 3rd Series. A Series of Modern
  Surveys in Mathematics]}.
\newblock Springer-Verlag, Berlin, second edition, 2005.

\bibitem{Ga}
W.~Gasch\"{u}tz.
\newblock Zu einem von {B}. {H}. und {H}. {N}eumann gestellten {P}roblem.
\newblock {\em Math. Nachr.}, 14:249--252 (1956), 1955.

\bibitem{Gruenberg}
K.~W. Gruenberg.
\newblock {\em Relation modules of finite groups}.
\newblock American Mathematical Society, Providence, R.I., 1976.
\newblock Conference Board of the Mathematical Sciences Regional Conference
  Series in Mathematics, No. 25.

\bibitem{LG}
F.~Grunewald and A.~Lubotzky.
\newblock Linear representations of the automorphism group of a free group.
\newblock {\em Geom. Funct. Anal.}, 18(5):1564--1608, 2009.

\bibitem{HarbaterSchneps}
D.~Harbater and L.~Schneps.
\newblock Approximating {G}alois orbits of dessins.
\newblock In {\em Geometric {G}alois actions, 1}, volume 242 of {\em London
  Math. Soc. Lecture Note Ser.}, pages 205--230. Cambridge Univ. Press,
  Cambridge, 1997.

\bibitem{Hartshorne1977}
R.~Hartshorne.
\newblock {\em Algebraic geometry}.
\newblock Springer-Verlag, New York-Heidelberg, 1977.
\newblock Graduate Texts in Mathematics, No. 52.

\bibitem{HidaModForms}
H.~Hida.
\newblock {\em Modular forms and {G}alois cohomology}, volume~69 of {\em
  Cambridge Studies in Advanced Mathematics}.
\newblock Cambridge University Press, Cambridge, 2000.

\bibitem{PropertyT}
M.~Kaluba, D.~Kielak, and P.~W. Nowak.
\newblock On property ({T}) for {$\mathrm{Aut}(F_n)$} and
  {$\mathrm{SL}_n(\mathbb{Z})$}.
\newblock {\em Ann. of Math. (2)}, 193(2):539--562, 2021.

\bibitem{PropertyT5}
M.~Kaluba, P.~W. Nowak, and N.~Ozawa.
\newblock {${\rm Aut}(\mathbb{F}_5)$} has property {$(T)$}.
\newblock {\em Math. Ann.}, 375(3-4):1169--1191, 2019.

\bibitem{LochakSchneps}
P.~Lochak and L.~Schneps.
\newblock Open problems in {G}rothendieck-{T}eichm\"{u}ller theory.
\newblock In {\em Problems on mapping class groups and related topics},
  volume~74 of {\em Proc. Sympos. Pure Math.}, pages 165--186. Amer. Math.
  Soc., Providence, RI, 2006.

\bibitem{LubotzkyProfinite}
A.~Lubotzky.
\newblock Pro-finite presentations.
\newblock {\em J. Algebra}, 242(2):672--690, 2001.

\bibitem{PropertyT4}
M.~Nitsche.
\newblock Computer proofs for {P}roperty {(T)}, and {SDP} duality.
\newblock Preprint, arXiv:2009.05134, 2020.

\bibitem{Zalesskii}
L.~Ribes and P.~Zalesskii.
\newblock {\em Profinite groups}, volume~40 of {\em Ergebnisse der Mathematik
  und ihrer Grenzgebiete. 3. Folge. A Series of Modern Surveys in Mathematics
  [Results in Mathematics and Related Areas. 3rd Series. A Series of Modern
  Surveys in Mathematics]}.
\newblock Springer-Verlag, Berlin, second edition, 2010.

\bibitem{SchnepsSurvey}
L.~Schneps.
\newblock The {G}rothendieck-{T}eichm\"{u}ller group {$\widehat{\rm GT}$}: a
  survey.
\newblock In {\em Geometric {G}alois actions, 1}, volume 242 of {\em London
  Math. Soc. Lecture Note Ser.}, pages 183--203. Cambridge Univ. Press,
  Cambridge, 1997.

\bibitem{SerreElliptic1968}
J.-P. Serre.
\newblock {\em Abelian {$l$}-adic representations and elliptic curves}.
\newblock McGill University lecture notes written with the collaboration of
  Willem Kuyk and John Labute. W. A. Benjamin, Inc., New York-Amsterdam, 1968.

\bibitem{GroupesAlgebriques}
J.-P. Serre.
\newblock {\em Algebraic groups and class fields}, volume 117 of {\em Graduate
  Texts in Mathematics}.
\newblock Springer-Verlag, New York, 1988.
\newblock Translated from the French.

\bibitem{ShimuraBook}
G.~Shimura.
\newblock {\em Introduction to the arithmetic theory of automorphic functions},
  volume~11 of {\em Publications of the Mathematical Society of Japan}.
\newblock Princeton University Press, Princeton, NJ, 1994.
\newblock Reprint of the 1971 original, Kano Memorial Lectures, 1.

\bibitem{Silverman}
J.~H. Silverman.
\newblock {\em The arithmetic of elliptic curves}, volume 106 of {\em Graduate
  Texts in Mathematics}.
\newblock Springer, Dordrecht, second edition, 2009.

\end{thebibliography}

\end{document}